\documentclass[12pt]{amsart}
\usepackage{amssymb}
\usepackage{mathrsfs}
\usepackage{graphicx}
\usepackage[all]{xy}
\renewcommand\a{\alpha}

\newcommand\g{\gamma}
\renewcommand\d{\delta}
\newcommand\la{\lambda}

\newcommand\s{\sigma}

\newcommand\vf{\varphi}

\renewcommand\t{\tau}

\newcommand\Om{\Omega}
\newcommand\w{\omega}

\newcommand\D{\Delta}

\newcommand\vL{\varLambda}

\newcommand\vG{\varGamma}
\newcommand\ve{\varepsilon}

\newcommand\N{\nabla }

\newcommand{\QQ}{\mathbb Q}

\newcommand{\ZZ}{\mathbb Z}

\newcommand{\CC}{\mathbb C}

\newcommand\Bm{\mathbf m}

\newcommand\bx{\mathbf{x}}
\newcommand\by{\mathbf{y}}

\newcommand\CA{\mathcal{A}}

\newcommand\CH{\mathcal{H}}

\newcommand\CO{\mathcal{O}}
\newcommand\CK{\mathcal{K}}

\newcommand\CF{\mathcal{F}}

\newcommand\CT{ \mathcal{T}}

\newcommand\FS{\mathfrak S}
\newcommand\Fu{\mathfrak u}
\newcommand\Fv{\mathfrak v}

\newcommand\Fs{\mathfrak s}

\newcommand\Ft{\mathfrak t}

\newcommand\Fgl{\mathfrak{gl}}
\newcommand\Fsl{\mathfrak{sl}}

\newcommand{\Bs}{\mathbf{s}}
\newcommand{\Bc}{\mathbf{c}}

\newcommand\wh{\widehat}
\newcommand\wt{\widetilde}

\newcommand{\lan}{\langle}
\newcommand{\ran}{\rangle}

\newcommand{\ra}{\rightarrow }

\newcommand{\xra}{\xrightarrow}

\newcommand\Ker{\operatorname{Ker}}
\newcommand\Hom{\operatorname{Hom}}
\newcommand\End{\operatorname{End}}

\newcommand\Ind{\operatorname{Ind}}
\newcommand\iInd{i\operatorname{-Ind}}
\newcommand\res{\operatorname{res}}
\newcommand\Res{\operatorname{Res}}
\newcommand\iRes{i\operatorname{-Res}}
\newcommand\coInd{\operatorname{coInd}}

\newcommand\Id{\operatorname{Id}}

\renewcommand\Im{\operatorname{Im}}

\newcommand\Std{\operatorname{Std}}
\newcommand\Ext{\operatorname{Ext}}

\newcommand{\KZ}{\operatorname{KZ}}

\newcommand{\cmod}{\operatorname{-mod}}
\newcommand{\cproj}{\operatorname{-proj}}

\newcommand{\HRes}{\,^\He \hspace{-1mm} \operatorname{Res}} 
\newcommand{\HInd}{\,^\He \hspace{-1mm} \operatorname{Ind}} 
\newcommand{\ORes}{\, ^{\CO} \hspace{-1mm} \Res}
\newcommand{\OInd}{\, ^{\CO} \hspace{-1mm} \Ind}
\newcommand{\pdim}{\operatorname{pdim}}

\newcommand{\isom}{\,\raise2pt\hbox{$\underrightarrow{\sim}$}\,}
\newcounter{ichi}
\newcommand{\roi}{\roman{ichi}}
\newcounter{ni}
\newcommand{\roii}{\roman{ni}}
\newcounter{san}
\newcommand{\roiii}{\roman{san}}
\newcounter{yon}
\newcommand{\roiv}{\roman{yon}}
\newcounter{go}
\newcommand{\rov}{\roman{go}}
\newcounter{roku}
\newcounter{nana}
\newcommand{\Sc}{\mathscr{S}}

\newcommand{\He}{\mathscr{H}}
\newcommand{\A}{\mathscr{A}}
\newcommand{\B}{\mathscr{B}}

\newcommand{\cas}{\circledast}

\newtheorem{thm}{Theorem}[section]
\newtheorem{lem}[thm]{Lemma}
\newtheorem{cor}[thm]{Corollary}
\newtheorem{prop}[thm]{Proposition}

\def \para{\refstepcounter{thm} \par\medskip\noindent
                \textbf{\thethm .} }

\def \remark{\refstepcounter{thm} \par\medskip\noindent
                \textbf{Remark \thethm .} }

\def \remarks{\refstepcounter{thm} \par\medskip\noindent
                \textbf{Remarks \thethm .} }

\allowdisplaybreaks[4]
\numberwithin{equation}{thm}

\begin{document}
\setlength{\baselineskip}{4.9mm}
\setlength{\abovedisplayskip}{4.5mm}
\setlength{\belowdisplayskip}{4.5mm}
\renewcommand{\theenumi}{\roman{enumi}}
\renewcommand{\labelenumi}{(\theenumi)}
\renewcommand{\thefootnote}{\fnsymbol{footnote}}
\renewcommand{\thefootnote}{\fnsymbol{footnote}}
\parindent=20pt
\newcommand{\dis}{\displaystyle}

\medskip
\begin{center}
{\large \bf Induction and Restriction Functors \\ for Cyclotomic $q$-Schur Algebras} 
\\
\vspace{1cm}
Kentaro Wada\footnote{This research was supported by 
JSPS Grant-in-Aid for Research Activity Start-up No. 22840025, 
and partially supported by 
GCOE \lq Fostering top leaders in mathematics', Kyoto University.}
\\[1em]

{\em Dedicated, with thanks, to Professor Toshiaki Shoji on the occasion of his retirement from
Nagoya University}

\address{Department of Mathematics, Faculty of Science, Shinshu University, 
	Asahi 3-1-1, Matsumoto 390-8621, Japan}
\email{wada@math.shinshu-u.ac.jp}

\end{center}

\title{}
\maketitle 
\markboth{Kentaro Wada}
{Induction and Restriction Functors  for Cyclotomic $q$-Schur Algebras}




\begin{abstract}
We define  the induction and restriction functors for cyclotomic $q$-Schur algebras, 
and study some properties of them. 
As an application, 
we categorify a higher level Fock space 
by using the module categories of cyclotomic $q$-Schur algebras. 
\end{abstract}


\setcounter{section}{-1}

\section{Introduction} 
Let $\He_{n,r}$ be the Ariki-Koike algebra associated to 
the complex reflection group $\FS_n \ltimes (\ZZ / r \ZZ)^n$ 
over a commutative ring $R$. 
Let 
$\Sc_{n,r}$ be the cyclotomic $q$-Schur algebra associated to $\He_{n,r}$. 
It is known that 
$\Sc_{n,r} \cmod$ is a highest weight cover of $\He_{n,r} \cmod$ 
in the sense of \cite{R} when $R$ is a field.  
In \cite{R}, 
Rouquier proved that 
$\Sc_{n,r} \cmod$ 
is equivalent to the category $\CO$ 
of the rational Cherednik algebra associated to $\FS_n \ltimes (\ZZ / r \ZZ)^n$ 
as the highest weight covers of $\He_{n,r} \cmod$ 
when $R=\CC$ with some special parameters.  

On the other hand, 
in \cite{BE}, Bezrukavnikov and Etingof 
defined the parabolic induction and restriction functors for rational Cherednik algebras. 
By using these functors, 
Shan has categorified a higher level Fock space by using the categories $\CO$ of rational Cherednik algebras 
in \cite{S}. 

In this paper, 
we define the induction and restriction functors for cyclotomic $q$-Schur algebras, 
and study some properties of them. 
In \S1, 
we review some  known results for cyclotomic $q$-Schur algebras. 
In \S2, 
we define the injective homomorphism of algebras 
$\iota : \Sc_{n,r} \ra \Sc_{n+1,r}$. 
This injection carries 
the unit element of $\Sc_{n,r}$ to a certain idempotent $\xi$ of $\Sc_{n+1,r}$. 
Thus, we can regard $\Sc_{n+1,r} \xi$ (resp. $\xi \Sc_{n+1,r}$) 
as an $(\Sc_{n+1,r}, \Sc_{n,r})$-bimodule (resp. $(\Sc_{n,r}, \Sc_{n+1,r})$-bimodule) 
by multiplications through the injection $\iota$.
By using these bimodules,  
we define the restriction functor from $\Sc_{n+1,r} \cmod$ to $\Sc_{n,r} \cmod$ by 
$\Res^{n+1}_n := \Hom_{\Sc_{n+1,r}} ( \Sc_{n+1,r} \xi, ?)$, 
and 
define the induction functors from $\Sc_{n,r} \cmod$ to $\Sc_{n+1,r} \cmod$ 
by 
$\Ind^{n+1}_n := \Sc_{n+1,r} \xi \otimes_{\Sc_{n,r}} ?$ 
and 
$\coInd^{n+1}_n := \Hom_{\Sc_{n,r}} (\xi \Sc_{n+1,r}, ?)$. 
In \S 3, 
we study the standard (Weyl) and costandard  modules of cyclotomic $q$-Schur algebras 
applying the functors $\Res^{n+1}_n$, $\Ind^{n+1}_n$ and $\coInd^{n+1}_n$. 
In Theorem \ref{Thm Res Ind Weyl}, 
we prove that the restricted and induced standard (resp. costandard) modules 
have filtrations whose successive quotients are isomorphic to standard (resp. costandard) modules. 
In \S4, 
we study some properties of our functors. 
In particular, 
we prove the isomorphism of functors 
$\Ind^{n+1}_n \cong \coInd^{n+1}_n$ 
(Theorem \ref{Thm iso Ind coInd} ). 
Then, we see that 
$\Res^{n+1}_n$ is left and right adjoint to $\Ind^{n+1}_n$, 
and both functors are exact. 
Moreover, 
these functors commute with the dual functors and Schur functors 
(Corollary \ref{Cor properties Res Ind}). 
In \S5, 
by using the projections to blocks of cyclotomic $q$-Schur algebras, 
we refine the induction and restriction functors. 
As an application, 
we categorify a level $r$ Fock space by using $\bigoplus_{n \geq 0} \Sc_{n,r} \cmod$ 
with (refined) induction and restriction functors 
(Corollary \ref{Cor categorification of Fock})
\footnote{Recently, in \cite{SW}, 
Stroppel and Webster gave a categorification of a Fock space by using a quiver Schur algebra, 
but our methods are totally different from theirs.}. 
In \S6, 
we prove that 
our induction and restriction functors are isomorphic to 
the corresponding parabolic induction and restriction functors for rational Cherednik algebras 
given in \cite{BE} 
when module categories of cyclotomic $q$-Schur algebras 
are equivalent to 
categories $\CO$ of rational Cherednik algebras 
as highest weight covers of module categories of Ariki-Koike algebras 
(Theorem \ref{Thm iso Res Ind SCA RCA}).

\quad

\textit{Notation and conventions:} 
For an algebra $\A$ over a commutative ring $R$,  
let $\A \cmod$ be  the category of finitely generated left $\A$-modules,  
and $K_0(\A\cmod)$ be the Grothendieck group of $\A \cmod$. 
For $M \in \A \cmod$, 
we denote by $[M]$ the image of $M$ in $K_0(\A \cmod)$.   

Let $\theta : \A \ra \A$ be an algebra anti-automorphism. 
For a left $\A$-module $M$, 
put $M^\cas = \Hom_R(M,R)$, 
and we define the left action of $\A$ on $M^\cas$ 
by 
$(a \cdot \vf) (m) = \vf ( \theta(a) \cdot m)$ for $a \in \A$, $\vf \in M^\cas$, $m \in M$. 
Then we have the contravariant functor 
$\cas : \A \cmod \ra \A \cmod$ such that $M \mapsto M^\cas$. 
Throughout this paper, 
we use the same symbol $\cas$ for contravariant functors defined in the above   
associated with several algebras  
since there is no risk to confuse.



\section{Review of cyclotomic $q$-Schur algebras} 
In this section, 
we recall the definition and some fundamental properties of the cyclotomic $q$-Schur algebra 
$\Sc_{n,r}$ 
introduced in \cite{DJM98}, 
and we review a presentation of $\Sc_{n,r}$ 
by generators and defining relations given in \cite{W}.

\para 
Let $R$ be a commutative ring,  
and 
we take parameters 
$q,Q_1,\dots, Q_r \in R$ 
such that 
$q$ is invertible in $R$. 
The Ariki-Koike algebra $_R \He_{n,r}$ associated to 
the complex reflection group  
$\FS_n \ltimes (\ZZ/r\ZZ)^n$ 
is the associative algebra with  1 over $R$ 
generated by 
$T_0,T_1,\dots,T_{n-1}$ 
with the following defining relations: 
\begin{align*}
&(T_0-Q_1)(T_0-Q_2)\dots (T_0-Q_r)=0, \\
&(T_i-q)(T_i+q^{-1})=0 &&(1 \leq i \leq n-1),\\
&T_0 T_1 T_0 T_1=T_1 T_0 T_1 T_0, \\
&T_i T_{i+1} T_i = T_{i+1} T_i T_{i+1} &&(1\leq i \leq n-2),\\
&T_i T_j=T_j T_i &&(|i-j|\geq 2). 
\end{align*}

The subalgebra of 
$_R \He_{n,r}$ 
generated by 
$T_1,\dots,T_{n-1}$ 
is isomorphic to 
the Iwahori-Hecke algebra 
$_R \He_n$ 
of the symmetric group  $\FS_n$ of degree $n$.  
For 
$w \in \FS_n$, 
we denote by 
$\ell(w)$ the length of $w$, 
and denote by 
$T_w$  the standard basis of $_R \He_n$ corresponding to $w$. 
Let 
$\ast : \, _R \He_{n,r} \ra \, _R \He_{n,r}$ 
($h \mapsto h^\ast$) 
be the anti-isomorphism given by 
$T_i^\ast = T_i$ 
for $i=0,1, \dots , n-1$.

\para 
Let $\Bm =(m_1,\dots, m_r) \in \ZZ_{>0}^r$ 
be an $r$-tuple of positive integers such that 
$m_k \geq n$ for any $k=1,\dots,r$. 
Put
\[ \vL_{n,r} (\Bm) =\left\{ \mu=(\mu^{(1)},\dots,\mu^{(r)}) \Biggm| 
	\begin{matrix}
		\mu^{(k)}=(\mu_1^{(k)},\dots,\mu_{m_k}^{(k)}) \in \ZZ_{\geq 0}^{m_k} \\
	    \sum_{k=1}^r \sum_{i=1}^{m_k} \mu_i^{(k)}=n 
	\end{matrix}	
	\right\}. 
\]
We denote by $|\mu^{(k)}|= \sum_{i=1}^{m_k} \mu_i^{(k)}$ (resp. $|\mu|=\sum_{k=1}^r |\mu^{(k)}|$) 
the size of $\mu^{(k)}$ (resp. the size of $\mu$), 
and call an element of $\vL_{n,r}(\Bm)$ an $r$-composition of size $n$. 
Put 
\[
\vL_{n,r}^+  = \{  \la \in \vL_{n,r}(\Bm)  \,|\, \la_1^{(k)} \geq \la_2^{(k)} \geq \dots \geq \la_{m_k}^{(k)} 
	\text{ for any } k=1,\dots,r \}.
\]
Then 
$\vL_{n,r}^+$ 
is the set of $r$-partitions of size $n$.

\para 
\label{Definition M^mu}
For $i=1,\dots,n$, 
put 
$L_1=T_0$ 
and  
$L_i=T_{i-1} L_{i-1} T_{i-1}$. 
For $\mu \in \vL_{n,r}(\Bm)$, 
put 
\[
m_\mu = \Big( \sum_{w \in \FS_\mu} q ^{\ell (w)} T_w \Big) \Big( \prod_{k=1}^r \prod_{i=1}^{a_k}(L_i - Q_k) \Big), 
\quad 
M^\mu = m_\mu \cdot \,_R \He_{n,r}, 
\] 
where 
$\FS_\mu$ is the Young subgroup of $\FS_n$ with respect to $\mu$, 
and 
$a_k=\sum_{j=1}^{k-1}|\mu^{(j)}|$ 
with 
$a_1=0$.  
The cyclotomic $q$-Schur algebra 
$_R \Sc_{n,r}$ 
associated to 
$_R \He_{n,r}$ 
is defined by 
\[ _
R \Sc_{n,r} 
= 
\,_R \Sc_{n,r} (\vL_{n,r}(\Bm)) 
= 
\End_{_R \He_{n,r}} \Big( \bigoplus_{\mu \in \vL_{n,r}(\Bm) } M^\mu \Big).
\]

\remark 
\label{Remark Morita equivalent} 
Let $\wt{\Bm} = (\wt{m}_1,\dots, \wt{m}_r) \in \ZZ^r_{>0}$ be 
such that 
$\wt{m}_k \geq n$ 
for any $k=1,\dots,r$.
Then 
it is known that 
$_R \Sc_{n,r}(\vL_{n,r}(\Bm))$ 
is Morita equivalent to 
$_R \Sc_{n,r}(\vL_{n,r}(\wt{\Bm}))$ 
when 
$R$ is a field.

\para 
In order to describe a presentation of $_R \Sc_{n,r}$,  
we prepare some notation. 

Put 
$m= \sum_{k=1}^r m_k$,  
and  let 
$P=\bigoplus_{i=1}^m \ZZ \ve_i$ 
be the weight lattice of $\Fgl_m$.
Set $\a_i=\ve_i - \ve_{i+1}$ for $i=1,\dots,m-1$, 
then 
$\Pi=\{\a_i\,|\, 1\leq i \leq m-1\}$ 
is the set of simple roots, 
and 
$Q=\bigoplus_{i=1}^{m-1} \ZZ\, \a_i$ 
is the root lattice of $\Fgl_m$. 
Put 
$Q^+ = \bigoplus_{i=1}^{m-1} \ZZ_{\geq 0}\, \a_i$. 
We define a partial order 
\lq\lq \,$ \geq$ "
on $P$, 
so called dominance order, 
by 
$\la \geq \mu $ if $\la - \mu \in Q^+$. 

Put 
$\vG(\Bm) = \{(i,k) \,|\, 1 \leq i \leq m_k, \, 1 \leq k \leq r\}$, 
and 
$\vG'(\Bm) = \vG(\Bm) \setminus \{(m_r , r)\}$. 
We  identify the set 
$\vG (\Bm) $  
with the set  
$\{1,\dots, m\}$
by the bijection 
\[
 \vG(\Bm) \ra \{1,\dots, m\} 
\text{ such that } 
(i,k) \mapsto \sum_{j=1}^{k-1} m_j + i.
\] 
Under this identification, 
we have 
\begin{align*}
&
P=\bigoplus_{i=1}^m \ZZ \ve_i = \bigoplus_{(i,k) \in \vG(\Bm)} \ZZ \ve_{(i,k)}, 
\\
&
Q=\bigoplus_{i=1}^{m-1} \ZZ\, \a_i = \bigoplus_{(i,k) \in \vG'(\Bm)} \ZZ \, \a_{(i,k)}.  
\end{align*}
Then 
we  regard 
$\vL_{n,r}(\Bm)$ 
as a subset of $P$ 
by the injective map 
\[ 
\vL_{n,r}(\Bm) 
\ra 
P 
\text{ such that } 
\la \mapsto \sum_{(i,k) \in \vG(\Bm)} \la_i^{(k)} \ve_{(i,k)}.
\] 
For convenience, 
we consider  $(m_k +1,k) = (1, k+1)$ for $(m_k,k) \in \vG'(\Bm)$ 
(resp. $(1-1,k)=(m_{k-1}, k-1)$ for  $(1, k) \in \vG(\Bm) \setminus \{(1,1)\}$).

For $i=1,\dots, n-1$, 
let 
$s_i=(i, i+1) \in \FS_n$ 
be the adjacent transposition. 
For $\mu \in \vL_{n,r}(\Bm)$ and $(i,k) \in \vG'(\Bm)$, 
put  
\begin{align*}
&
X_{\mu}^{\mu + \a_{(i,k)}} 
=  
\{ 1, s_{N}, s_{N} s_{N-1}, \dots, s_{N}s_{N-1} \dots s_{N - \mu_{i}^{(k)} +1} \}, 
\\
&
X_\mu^{\mu - \a_{(i,k)}} 
= 
\{ 1, s_{N}, s_{N} s_{N+1}, \dots, s_{N} s_{N+1} \dots s_{N + \mu_{i+1}^{(k)} -1} \},
\end{align*}
where 
$N= \sum_{l=1}^{k-1}|\mu^{(l)}| + \sum_{j=1}^i \mu_j^{(k)}$, 
and put 
$\mu_{m_k +1}^{(k)} = \mu_1^{(k+1)}$ if $i = m_k$.

For 
$(i,k) \in \vG'(\Bm)$, 
we define the elements 
$E_{(i,k)}, F_{(i,k)} \in \, _R \Sc_{n,r}$ by 
\begin{align}
\label{description E(i,k)}
&
E_{(i,k)} (m_\mu \cdot h) 
= 
\begin{cases} 
\dis 
q^{- \mu_{i+1}^{(k)} +1}  \Big( \sum_{ x \in X_{\mu}^{\mu + \a_{(i,k)}}} q^{\ell(x)} T_x^\ast \Big) 
h^\mu_{+(i,k)} m_\mu  
\cdot h 
& \text{ if } \mu + \a_{(i,k)} \in \vL_{n,r}(\Bm),
\\
0 
& \text{ if } \mu + \a_{(i,k)} \not\in \vL_{n,r}(\Bm),
\end{cases}
\\
\label{description F(i,k)}
&
F_{(i,k)} (m_\mu \cdot h) 
= 
\begin{cases} 
\dis 
q^{- \mu_{i}^{(k)} +1} 
	\Big( \sum_{ y \in X_{\mu}^{\mu - \a_{(i,k)}}} q^{\ell(y)} T_y^\ast \Big)
m_\mu  
\cdot h 
& \text{ if } \mu - \a_{(i,k)} \in \vL_{n,r}(\Bm), 
\\
0 
& \text{ if } \mu - \a_{(i,k)} \not\in \vL_{n,r}(\Bm) 
\end{cases}
\end{align}
for 
$\mu \in \vL_{n,r}(\Bm)$ and  
$h \in \, _R \He_{n,r}$, 
where 
$h^\mu_{+(i,k)} = 
\begin{cases} 
1 & (i \not= m_k),
\\
L_{N+1} - Q_{k+1} 
& ( i= m_k).
\end{cases} 
$

For 
$\la \in \vL_{n,r}(\Bm)$, 
we define the element $1_\la \in \, _R \Sc_{n,r}$ by 
\[ 
1_\la (m_\mu \cdot h) 
= 
\d_{\la,\mu} m_\la \cdot h 
\]
for 
$\mu \in \vL_{n,r}(\Bm)$ and  
$h \in \, _R \He_{n,r}$. 
From this definition, 
we see that 
$\{ 1_\la \,|\, \la \in \vL_{n,r}(\Bm)\}$ 
is a set of pairwise orthogonal idempotents, 
and 
we have 
$1 = \sum_{\la \in \vL_{n,r}(\Bm)} 1_\la$.

For 
$\la \in \vL_{n,r}(\Bm)$ 
and 
$(i,k) \in\nobreak \vG(\Bm)$, 
we define 
$\s_{(i,k)}^\la \in \, _R \Sc_{n,r}$ by 
\[ 
\s_{(i,k)}^\la ( m_\mu \cdot h) 
= \begin{cases} 
\d_{\la,\mu} \big( m_\la (L_{N+1} + L_{N+2} + \dots + L_{N+\la_i^{(k)}}) \big) \cdot h 
& \text{ if } \la_i^{(k)} \not=0, 
\\
0
& \text{ if } \la_i^{(k)} =0
\end{cases}
\] 
for $\mu \in \vL_{n,r}(\Bm)$ and $ h \in \, _R \He_{n,r}$, 
where 
$N=\sum_{l=1}^{k-1} | \la^{(l)}| + \sum_{j=1}^{i-1} \la_j^{(k)}$.
For $(i,k) \in\nobreak \vG(\Bm)$, 
put 
\[
\s_{(i,k)} = \sum_{\la \in \vL_{n,r}(\Bm)} \s_{(i,k)}^\la,
\] 
then 
$\s_{(i,k)} $ 
is a Jucys-Murphy element 
of $_R \Sc_{n,r}$ 
(See \cite{Mat08} for properties of Jucys-Murphy elements).

\para 
We need some non-commutative polynomials to described a presentation of $_R \Sc_{n,r}$ 
as follows. 
Put  
$\CA = \ZZ [ q,q^{-1}, Q_1,\dots, Q_r]$, 
where 
$q,Q_1,\dots,Q_r$ are indeterminate  over $\ZZ$, 
and let 
$\CK = \QQ (q, Q_1,\dots,Q_r)$ 
be the quotient field of $\CA$.

Let 
$\CK\lan \bx \ran $ 
(resp. $\CK \lan \by \ran$) 
be the non-commutative polynomial ring over $\CK$ 
with indeterminate variables 
$\bx = \{ x_{(i,k)} \,|\, (i,k) \in \vG'(\Bm)\}$ 
(resp. $\by = \{ y_{(i,k)} \,|\, (i,k) \in\nobreak \vG'(\Bm)\}$). 
For 
$g(\bx) \in \CK\lan \bx \ran$ 
(resp. $g(\by) \in \CK\lan \by \ran$), 
let 
$g(F)$ (resp. $g(E)$) 
be the element of $_\CK \Sc_{n,r}$ 
obtained by replacing 
$x_{(i,k)}$ (resp. $y_{(i,k)}$) with $E_{(i,k)}$ (resp. $F_{(i,k)}$). 
Moreover, 
for
 $g(\bx,\by) = \sum_j r_j g_j^-(\bx) \otimes g_j^+(\by) 
	\in \CK\lan \bx \ran \otimes_{\CK} \CK \lan \by \ran$ 
($r_j \in \CK$), put 
\[ 
g(F,E) = \sum_{j} r_j g_j^-(F) g_j^+(E) \in \, _\CK \Sc_{n,r}.  
\]
Then we have the following lemma. 
\begin{lem}[{\cite[Lemma 7.2]{W}}]\
\label{Lemma polynomial JM} 
For $\la \in \vL_{n,r}(\Bm)$ and $(i,k) \in \vG(\Bm)$, 
there exists a (non-commutative) polynomial 
$g^{\la}_{(i,k)} (\bx, \by) \in \CK \lan \bx \ran \otimes_{\CK} \CK \lan \by \ran$ 
such that 
\begin{align}
\label{polynomial JM} 
\s^\la_{(i,k)} = g^\la_{(i,k)} (F,E) 1_\la 
\end{align}
in $_\CK \Sc_{n,r}$. 
\end{lem}

We remark that 
a polynomial $g^{\la}_{(i,k)} (\bx, \by) \in \CK \lan \bx \ran \otimes_{\CK} \CK \lan \by \ran$  
satisfying \eqref{polynomial JM} 
is not unique in general. 
Then 
we take and fix a polynomial $g^{\la}_{(i,k)} (\bx, \by) \in \CK \lan \bx \ran \otimes_{\CK} \CK \lan \by \ran$  
 ($\la \in \vL_{n,r}(\Bm)$, $(i,k) \in \vG(\Bm)$) 
satisfying \eqref{polynomial JM} 
when we describe a presentation of $_\CK \Sc_{n,r}$.

For an integer $k \in \ZZ$, 
put 
$[k] = (q^k - q^{-k})/(q - q^{-1})$. 
For a positive integer $t \in \ZZ_{>0}$, 
put 
$[t]!=[t][t-1] \dots [1]$ 
and 
set 
$[0]!=1$. 

Now we can describe a presentation of cyclotomic $q$-Schur algebras 
as follows. 

\begin{thm}[{\cite[Theorem 7.16]{W}}]\
\label{Theorem presentation}
$_\CK \Sc_{n,r}$ is the associative algebra over $\CK$ generated by  
$E_{(i,k)}$, $F_{(i,k)}$ $((i,k) \in \vG'(\Bm))$, $1_\la$ $(\la \in \vL_{n,r}(\Bm))$ 
with the following defining relations:  
\begin{align}
&1_\la 1_\mu = \d_{\la,\mu} 1_\la, \quad \sum_{\la \in \vL_{n,r}(\Bm)} 1_\la =1, \label{S-1}\\
&E_{(i,k)} 1_\la = 
	\begin{cases} 
		1_{\la + \a_{(i,k)}} E_{(i,k)} & \text{ if }\la + \a_{(i,k)} \in \vL_{n,r}(\Bm), \label{S-2}\\
		0 & \text{otherwise},
	\end{cases}	\\
&F_{(i,k)} 1_\la = 
	\begin{cases} 
		1_{\la - \a_{(i,k)}} F_{(i,k)} & \text{ if }\la - \a_{(i,k)} \in \vL_{n,r}(\Bm), \label{S-3}\\
		0 & \text{otherwise},
	\end{cases}	\\
&1_{\la} E_{(i,k)}  = 
	\begin{cases} 
		E_{(i,k)} 1_{\la - \a_{(i,k)}} & \text{ if }\la - \a_{(i,k)} \in \vL_{n,r}(\Bm), \label{S-4}\\
		0 & \text{otherwise},
	\end{cases}	\\
&1_\la F_{(i,k)} = 
	\begin{cases} 
		F_{(i,k)} 1_{\la + \a_{(i,k)}} & \text{ if }\la + \a_{(i,k)} \in \vL_{n,r}(\Bm), \label{S-5}\\
		0 & \text{otherwise},
	\end{cases}	\\
&
\label{S-6}
E_{(i,k)}F_{(j,l)} - F_{(j,l)}E_{(i,k)} 
		=\d_{(i,k),(j,l)} \sum_{\la \in \vL_{n,r}} \eta_{(i,k)}^\la, 
\end{align}
where
$ \displaystyle 
\eta_{(i,k)}^\la = 
\begin{cases} 
	[\la_i^{(k)} - \la_{i+1}^{(k)}] 1_\la 
		&\text{if } i\not= m_k, 
	\\[3mm]
	\Big( -Q_{k+1} [ \la_{m_k}^{(k)} - \la_1^{(k+1)}] 
	\\ \quad 
		+ q^{\la_{m_k}^{(k)} - \la_1^{(k+1)}} 
			(q^{-1} g_{(m_k,k)}^\la(F,E) - q g_{(1,k+1)}^\la(F,E) ) \Big) 1_\la 
		&\text{if } i=m_k,
\end{cases}
$
\begin{align}
&E_{(i \pm 1,k)}(E_{(i,k)})^2 - (q+q^{-1}) E_{(i,k)} E_{(i \pm 1,k)} E_{(i,k)} 
	+ (E_{(i,k)})^2 E_{(i \pm 1,k)}=0 , \label{S-7}\\
& E_{(i,k)} E_{(j,l)}= E_{(j,l)} E_{(i,k)} \qquad (|(i,k)- (j,l)| \geq 2), \notag \\
&F_{(i \pm 1,k)}(F_{(i,k)})^2 - (q+q^{-1}) F_{(i,k)} F_{(i \pm 1,k)} F_{(i,k)} 
	+ (F_{(i,k)})^2 F_{(i \pm 1,k)}=0,  \label{S-8}\\
& F_{(i,k)} F_{(j,l)}= F_{(j,l)} F_{(i,k)} \qquad  (|(i,k)- (j,l)| \geq 2), \notag 
\end{align}
where 
$(i,k) - (j,l) = (\sum_{a=1}^{k-1} m_a + i ) - ( \sum_{b=1}^{l-1} m_b +j)$ 
for $(i,k), (j,l) \in \vG'(\Bm)$. 

Moreover, 
$_\CA \Sc_{n,r}$ is isomorphic to the $\CA$-subalgebra of $_\CK \Sc_{n,r}$ 
generated by 
$E_{(i,k)}^l/[l]!$, 
$F_{(i,k)}^l/[l]!$ 
$\big( (i,k) \in \vG'(\Bm),  \, l \geq 1\big)$, 
$1_\la$ $(\la \in \vL_{n,r}(\Bm) )$. 
Then 
we can obtain the cyclotomic $q$-Schur algebra 
$_R \Sc_{n,r}$ over $R$ 
as the specialized algebra  $R \otimes_{\CA} \, _\CA \Sc_{n,r}$ of $_\CA \Sc_{n,r}$. 
\end{thm}

\para 
\label{Definition Weyl}
\textit{Weyl modules} (see \cite{W} for more details).
Let 
$_\CA \Sc_{n,r}^+$ (resp. $_\CA \Sc_{n,r}^-$) 
be the subalgebra of $_\CA \Sc_{n,r}$ 
generated by 
$E_{(i,k)}^{l}/[l]!$ 
(resp. $F_{(i,k)}^{l}/[l]!$) 
for $(i,k) \in \vG'(\Bm)$ and $l \geq 1$. 
Let  
$_\CA \Sc_{n,r}^0$ 
be the subalgebra of $_\CA \Sc_{n,r}$ 
generated by 
$1_\la$ for $\la \in \vL_{n,r}(\Bm)$. 
Then 
$_\CA \Sc_{n,r}$ has the triangular decomposition 
$_\CA \Sc_{n,r} = \, _\CA \Sc_{n,r}^- \, _\CA \Sc_{n,r}^0 \, _\CA \Sc_{n,r}^+$ 
by \cite[Proposition 3.2, Theorem 4.12, Theorem 5.6, Proposition 6.4, Proposition 7.7 and Theorem 7.16]{W}
. 
We denote by 
$_\CA \Sc_{n,r}^{\geq 0}$ 
the subalgebra of $_\CA \Sc_{n,r}$ 
generated by 
$_\CA \Sc_{n,r}^+$ and $_\CA \Sc_{n,r}^0$. 

Note that 
$_R \Sc_{n,r}$ 
is the specialized algebra 
$R \otimes_{\CA} \, _\CA \Sc_{n,r}$. 
We denote by 
$E^{(l)}_{(i,k)}$ (resp. $F^{(l)}_{(i,k)}$, $1_\la$) 
the elements $1 \otimes E^l_{(i,k)}/[l]!$ (resp. $1 \otimes F^l_{(i,k)}/[l]!$, $1 \otimes 1_\la$) 
of $R \otimes_{\CA} \, _\CA \Sc_{n,r}$. 
Then $_R \Sc_{n,r}$ also has the triangular decomposition 
\[
_R \Sc_{n,r} = \, _R \Sc_{n,r}^- \, _R \Sc_{n,r}^0 \, _R \Sc_{n,r}^+
\] 
which comes from the triangular decomposition of $_\CA \Sc_{n,r}$.

For $\la \in \vL_{n,r}^+$, 
we define the one-dimensional $_R \Sc_{n,r}^{\geq 0}$-module 
$\Theta_\la = R v_\la$ 
by 
$E_{(i,k)}^{(l)} \cdot\nobreak  v_\la =0$ $\big( (i,k) \in \vG'(\Bm), l \geq 1 \big)$ 
and 
$1_\mu \cdot v_\la = \d_{\la,\mu} v_\la $ $(\mu \in \vL_{n,r}(\Bm))$. 
Then the  Weyl module $_R \D_n(\la)$ of $_R \Sc_{n,r}$ 
is defined as the induced module of $\Theta_\la$:
\[ 
_R \D_n(\la) = \, _R \Sc_{n,r} \otimes_{_R \Sc_{n,r}^{\geq 0}} \Theta_\la. 
\]
See also \cite[paragraph 3.3 and Theorem 3.4]{W} for definitions of $_R \D_n(\la)$.

It is known that 
$_\CK \Sc_{n,r}$ 
is semi-simple,  
and that 
$\{_\CK \D_n (\la) \,|\, \la \in \vL_{n,r}^+\}$ 
gives a complete set of pairwise non-isomorphic (left) simple $_\CK \Sc_{n,r}$-modules.

\para 
\textit{Highest weight modules.} 
Let 
$M$ be an $_R \Sc_{n,r}$-module. 
We say that 
an element $m \in M$ 
is a primitive  vector if 
$E_{(i,k)}^{(l)} \cdot m=0$ 
for any $(i,k) \in \vG'(\Bm)$ and $l \geq 1$,  
and say that 
$m \in M$ 
is a weight vector of weight $\mu$ 
if $1_\mu \cdot m =m$ for $\mu \in \vL_{n,r}(\Bm)$. 
If 
$x_\la \in M$ 
is a primitive and  a weight vector of weight $\la$, 
we say that 
$x_\la$ is a highest weight vector of weight $\la$. 
If 
$M$ is generated by a highest weight vector $x_\la \in M$ of weight $\la$ 
as an $_R \Sc_{n,r}$-module, 
we say that 
$M$ is a highest weight module of highest weight $\la$. 
It is clear that 
the Weyl module $_R \D_n(\la)$ ($\la \in \vL_{n,r}^+$) 
is a highest weight module. 
Moreover, 
we have the following universality of the Weyl modules.

\begin{lem}
\label{Lemma universality of Weyl} 
Let $M$ be an $_R \Sc_{n,r}$-module. 
If $M$ is a highest weight module of highest weight $\la$, 
there exists a surjective $_R \Sc_{n,r}$-homomorphism 
$_R \D_n(\la) \ra M$ 
such that 
$1 \otimes v_\la \mapsto x_\la$, 
where 
$x_\la$ is a highest weight vector of $M$. 
\end{lem}

\begin{proof}
Let $x_\la \in M$ be a highest weight vector of weight $\la$. 
Then we can define the well-defined $_R \Sc_{n,r}^{\geq 0}$-balanced map 
$_R \Sc_{n,r} \times \Theta_\la \ra M $ 
such that  
$(s, v_\la) \mapsto s \cdot x_\la$. 
This implies the $_R \Sc_{n,r}$-homomorphism 
$_R \Sc_{n,r} \otimes_{_R \Sc_{n,r}^{\geq 0}} \Theta_\la \ra M$ 
such that 
$s \otimes v_\la \mapsto s \cdot x_\la$. 
Since 
$M$ is generated by $x_\la$ as an $_R \Sc_{n,r}$-module, 
this homomorphism is surjective. 
\end{proof}

\para 
In \cite{DJM98}, 
it was proven that 
$_R \He_{n,r}$ (resp. $_R \Sc_{n,r}$) is a cellular algebra 
by using combinatorial arguments.  
We will use such structures and combinatorics in the later arguments. 
So, we review some of them (see \cite{DJM98} for details).

For $\mu \in \vL_{n,r}(\Bm)$, 
the diagram $[\mu]$ of $\mu$ is 
the set 
\[
[\mu]= 
\{(i,j,k) \in \ZZ^3 \,|\, 1 \leq i \leq m_k, \, 1 \leq j \leq \mu_i^{(k)}, \, 1 \leq k \leq r\}.
\] 
For $\la \in \vL_{n,r}^+$ and $x \in \ZZ_{>0} \times \ZZ_{>0} \times \{1, \dots, r\}$, 
we say that 
$x$ is a removable node (resp. an addable node) 
of $\la$ 
if 
$[\la] \setminus \{x\}$ (resp. $[\la] \cup \{x\}$) 
is the diagram of a certain $r$-partition $\mu \in \vL_{n-1,r}^+$ 
(resp. $\mu \in \vL_{n+1,r}^+$). 
In such case, 
we denote the above  $\mu \in \vL_{n-1,r}^+$ (resp.  $\mu \in \vL_{n+1,r}^+$) 
by $\la \setminus x$ (resp. $\la \cup x$), 
namely 
$[\la \setminus x]=[\la] \setminus\{x\}$ 
(resp. $[\la \cup x]=[\la] \cup \{x\}$).

We define the partial order \lq\lq $\succeq$" on $\ZZ_{>0}\times \ZZ_{>0} \times \{1,\dots, r\}$ 
by 
\[ 
(i,j,k) \succ (i', j', k') 
\text{ if }
k< k',  
\text{ or if } 
k=k' \text{ and } i<i'.
\]
We also define a partial order \lq\lq $\succeq$" on $\ZZ_{>0} \times \{1,\dots, r\}$ 
by 
\[
(i,k) \succ (i', k') 
\text{ if } 
(i,1,k) \succ (i',1,k').
\] 
For $\la \in \vL^+_{n,r}$, 
a standard tableau $\Ft$ of shape $\la$ 
is a bijection 
\[ 
\Ft : [\la] \ra \{1, 2, \dots, n\} 
\] 
satisfying the following two conditions: 
\begin{enumerate}
\item 
$\Ft((i,j,k)) < \Ft((i, j+1,k))$ if $(i,j+1, k) \in [\la]$, 

\item 
$\Ft((i,j,k)) < \Ft((i+1, j,k))$ if $(i+1,j, k) \in [\la]$. 
\end{enumerate}
We denote by $\Std (\la)$ the set of standard tableaux of shape $\la$.

For $\mu \in \vL_{n,r}(\Bm)$, 
we define the bijection 
$\Ft^\mu : [\mu] \ra \{1,2,\dots, n\}$ as  
\[ 
\Ft^\mu ((i,j,k)) = \sum_{c=1}^{k-1} |\mu^{(c)}| + \sum_a^{i-1} \mu^{(k)}_a + j.
\]
It is clear that 
$\Ft^\la \in \Std(\la)$ 
for $\la \in \vL_{n,r}^+$. 
For $\Ft \in \Std(\la)$, 
we define $d(\Ft) \in \FS_n$ as 
\[ 
\Ft((i,j,k)) = d(\Ft) (\Ft^\la (i,j,k)) \quad ((i,j,k) \in [\la]).
\]
For 
$\la \in \vL_{n,r}^+$ and $\Fs, \Ft \in \Std(\la)$, 
set   
$m_{\Fs \Ft} = T^\ast_{d(\Fs)} m_\la T_{d(\Ft)} \in \, _R \He_{n,r}$.  
Then we have the following theorem. 

\begin{thm}[{\cite[Theorem 3.26]{DJM98}}]\
$_R \He_{n,r}$ is a cellular algebra with a cellular basis 
$\{m _{\Fs \Ft} \,|\, \Fs, \Ft \in \Std(\la) \text{ for some } \la \in \vL_{n,r}^+\}$ 
with respect to the poset $(\vL_{n,r}^+, \geq)$. 
In particular, 
we have 
$m_{\Fs \Ft}^\ast = m_{\Ft \Fs}$. 
\end{thm}

\para 
For $\la \in \vL_{n,r}^+$ and $\mu \in \vL_{n,r}(\Bm)$, 
a tableau of shape $\la$ with weight $\mu$ 
is a map 
\[
T :\nobreak [\la] \ra \{(a,c) \in \ZZ \times \ZZ \,|\, a \geq 1, 1 \leq c \leq r\}
\] 
such that 
$\mu_i^{(k)} = \sharp \{x \in [\la] \,|\, T(x) = (i,k)\}$. 
We define the order on $\ZZ \times \ZZ$ 
by 
$(a,c) \geq (a',c')$ either if $c > c'$, or if $c=c'$ and $a \geq a'$.
For a tableau $T$ of shape $\la$ with weight $\mu$, 
we say that $T$ is  semi-standard if $T$ satisfies the following conditions: 
\begin{enumerate}
\item If $T((i,j,k)) =(a,c)$, then $k \leq c$,
\item $T((i,j,k)) \leq T((i,j+1,k))$ if $(i,j+1,k) \in [\la]$, 
\item $T((i,j,k)) < T((i+1,j,k))$ if $(i+1,j,k) \in [\la]$. 
\end{enumerate}
For $\la \in \vL_{n,r}^+$ and  $\mu \in \vL_{n,r}(\Bm)$, 
we denote by 
$\CT_0(\la,\mu)$ 
the set of semi-standard tableaux of shape $\la$ with weight $\mu$. 
Put $\CT_0(\la) = \bigcup_{\mu \in \vL_{n,r}(\Bm)} \CT_0(\la,\mu)$.

For $\la \in \vL^+_{n,r}$, 
let 
$T^\la$ be the tableau of shape $\la$ with weight $\la$ 
such that 
\[ 
T^\la((i,j,k)) =(i,k).
\]
It is clear that 
$T^\la$ is semi-standard, 
and it is the unique semi-standard tableau of shape $\la$ with weight $\la$. 
Namely, we have $\CT_0(\la,\la) =\{ T^\la\}$.

For $\Ft \in \Std(\la)$ ($\la \in \vL_{n,r}^+$) and $\mu \in \vL_{n,r}$, 
we define the tableau $\mu(\Ft)$ of shape $\la$ with weight $\mu$ by 
\[
\mu(\Ft) ((i,j,k)) = (a,c), 
\text{ if } 
\Ft^\mu ((a,b,c)) = \Ft((i,j,k)) 
\text{ for some $b$}.
\]
For 
$S\in \CT_0(\la,\mu)$, $T \in \CT_0(\la,\nu)$ ($\la \in \vL_{n,r}^+$, $\mu, \nu \in \vL_{n,r}(\Bm)$), 
put 
\[
m_{ST} = \sum_{\Fs, \Ft \in \Std(\la) \atop \mu(\Fs) =S, \nu(\Ft)=T} q^{\ell(d(\Fs)) + \ell(d(\Ft))} 
	m_{\Fs \Ft}, 
\]
and  define the element $\vf_{ST} \in \, _R \Sc_{n,r}$ by 
\[ 
\vf_{ST}( m_\t \cdot h) = \d_{\nu, \t} m_{ST} \cdot h 
	\quad ( \t \in \vL_{n,r}(\Bm), \, h \in \, _R \He_{n,r}).
\]
Then we have the following theorem.

\begin{thm}[{\cite[Theorem 6.6]{DJM98}}]\
$_R \Sc_{n,r}$ is a cellular algebra with a cellular basis 
$\{\vf_{ST}\,|\, S, T \in \CT_0(\la) \text{ for some } \la \in \vL_{n,r}^+\}$ 
with respect to the poset $(\vL_{n,r}^+, \geq)$. 
In particular, 
there exists an anti-automorphism  
$\theta_n : \,_R \Sc_{n,r} \ra \, _R \Sc_{n,r}$ 
such that 
$\theta_n(\vf_{ST})=\vf_{TS}$. 
Moreover, 
$_R\Sc_{n,r}$ 
is a quasi-hereditary algebra 
when 
$R$ is a field. 
\end{thm}

For the anti-automorphism $\theta_n$ of $_R \Sc_{n,r}$, 
we have the following lemma. 

\begin{lem}
\label{Lemma image theta}
For $\mu \in \vL_{n,r}(\Bm)$, $(i,k) \in \vG'(\Bm)$ and $l \geq 1$, 
we have that 
\begin{align}
&\theta_n (1_\mu ) = 1_\mu, 
\\
&\theta_n (E_{(i,k)}^{(l)} 1_\mu ) = q^{l \cdot (\mu_i^{(k)} - \mu_{i+1}^{(k)} +  l ) }  
1_\mu F_{(i,k)}^{(l)}, 
\label{theta E}
\\
&\theta_n (1_\mu F_{(i,k)}^{(l)}  ) = q^{- l \cdot (\mu_i^{(k)} - \mu_{i+1}^{(k)} +  l ) }  
E_{(i,k)}^{(l)} 1_\mu.
\label{theta F} 
\end{align}
\end{lem}

\begin{proof}
Note that 
$\theta_n$ on $_\CA \Sc_{n,r}$ 
is 
the restriction to $_\CA \Sc_{n,r}$ of $\theta_n$ on $_\CK \Sc_{n,r}$, 
and 
it is enough to show the case where $R = \CK$ by using the argument of specialization.  
Moreover, 
it is enough to show the case where $l=1$ 
since 
we can obtain the statements for $l \geq 2$ 
by the inductive arguments thanks to  the equation  
\begin{align*}
\theta_n(E^{(l)}_{(i,k)} 1_\mu) 
&= 
\theta_n \big( 1/ [l]  (E_{(i,k)}^{(l-1)} 1_{\mu + \a_{(i,k)}}) (E_{(i,k)} 1_\mu) \big)
\\
&= 
1/[l] \theta_n(E_{(i,k)} 1_\mu) \theta_n(E_{(i,k)}^{(l-1)} 1_{\mu + \a_{(i,k)}}). 
\end{align*} 

From the definitions, 
we see that 
$1_\la = \vf_{T^\la T^\la}$, 
and we obtain  
$\theta_n (1_\la) = 1_\la$. 
By 
\cite[paragraphs 5.5, 6.2, and Lemma 6.10]{W}, 
we see that 
\begin{align*}
E_{(i,k)}(m_\mu) 
=
q^{\mu_i^{(k)} - \mu_{i+1}^{(k)} +1} \big( F_{(i,k)} (m_{\mu + \a_{(i,k)}}) \big)^\ast 
\end{align*}
Combining with  \cite[Proposition 6.9 and Lemma 6.10]{DJM98}, 
we have that 
\begin{align*}
\theta_n (E_{(i,k)} 1_\mu ) (m_{\mu + \a_{(i,k)}}) 
&= 
\big( E_{(i,k)} (m_\mu) \big)^\ast 
\\
&= 
q^{\mu_i^{(k)} - \mu_{i+1}^{(k)} +1} 
\big( (F_{(i,k)} (m_{\mu + \a_{(i,k)}}))^\ast \big)^\ast 
\\
&= 
q^{\mu_i^{(k)} - \mu_{i+1}^{(k)} +1}  F_{(i,k)} (m_{\mu + \a_{(i,k)}}), 
\end{align*}
and $\theta_n (E_{(i,k)} 1_\mu) (m_\t) =0$ unless $ \t = \mu + \a_{(i,k)}$. 
Thus, we have that 
\[\theta_n (E_{(i,k)} 1_\mu ) 
= q^{\mu_i^{(k)} - \mu_{i+1}^{(k)} +1} F_{(i,k)} 1_{\mu + \a_{(i,k)}} 
= q^{\mu_i^{(k)} - \mu_{i+1}^{(k)} +1} 1_\mu F_{(i,k)}.
\] 
Now we proved \eqref{theta E}, 
and,  
by applying $\theta_n$ to the equation \eqref{theta E}, 
we obtain \eqref{theta F}.
\end{proof}

\para 
From the definitions,  
we have that 
\begin{align}
\label{weight celluar basis of Sc}
1_\t \vf_{TS} = \d_{\mu,\t} \vf_{TS}  
\text{ and }  
\vf_{TS}  1_\t = \d_{\nu, \t} \vf_{TS}
\end{align} 
for $T \in \CT_0(\la,\mu)$, $S \in \CT_0(\la,\nu)$  
(see \cite{DJM98}). 

For $\la \in \vL_{n,r}^+$, 
let 
$_R \Sc_{n,r}(> \la)$ 
be the $R$-submodule of $_R \Sc_{n,r}$ 
spanned by 
\[ 
\{ \vf_{ST} \,|\, S, T \in \CT_0(\la') \text{ for some } \la' \in \vL_{n,r}^+ \text{ such that } \la' > \la \},  
\]
and 
$_R W(\la)$ 
be the $R$-submodule of $_R \Sc_{n,r} / \, _R \Sc_{n,r}(>\la)$ 
spanned by 
\[ 
\{ \vf_{T T^\la} + \, _R \Sc_{n,r}(>\la) \,|\, T \in \CT_0(\la) \}. 
\]
Then, 
thanks to the general theory of cellular algebras, 
$_R\Sc_{n,r}(>\la)$ 
turns out to be a two-sided ideal of 
$_R \Sc_{n,r}$, 
and 
$_R W(\la)$ 
turns out to be an $_R \Sc_{n,r}$-submodule 
of $_R \Sc_{n,r}/ \, _R \Sc_{n,r}(>\la)$ 
whose action comes from the multiplication of $_R \Sc_{n,r}$. 
Put 
$\vf_{T} = \vf_{T T^\la} + \, _R \Sc_{n,r}(> \la)$ for $T \in \CT_0(\la)$, 
then 
$\{ \vf_T \,|\, T \in \CT_0(\la) \}$ 
gives an $R$-free basis of 
$_R W(\la)$, 
and 
it is known that 
$_R W(\la)$ is generated by 
$\vf_{T^\la}$ 
as 
an $_R \Sc_{n,r}$-module.

\begin{lem}
\label{Lemma iso D W}
For each $\la \in \vL_{n,r}^+$, 
there exists the $_R \Sc_{n,r}$-isomorphism 
\[ 
_R \D_n (\la) \ra \, _R W(\la) 
\text{ such that } 
1 \otimes v_\la \mapsto \vf_{T^\la}.  
\]
\end{lem}
\begin{proof}
From the definition of semi-standard tableaux, 
we see that $\la \geq \mu$ if $\CT_0(\la,\mu) \not=\nobreak \emptyset$. 
Thus, we have that 
$1_\mu \cdot \, _R W(\la) =0 $  unless $ \la \geq \mu$. 
On the other hand, 
we have that 
\[
E_{(i,k)}^{(l)} \cdot \vf_{T^\la} 
= 
E_{(i,k)}^{(l)} 1_\la \cdot \vf_{T^\la} 
= 
1_{\la + l \cdot \a_{(i,k)}} E_{(i,k)} \cdot \vf_{T^\la} 
\] 
for $(i,k) \in \vG'(\Bm)$ and $l \geq 1$. 
Note that $\la + l \cdot \a_{(i,k)} > \la$,  
these imply that 
\[
E_{(i,k)}^{(\l)} \cdot \vf_{T^\la} =0 
\text{ for any } 
(i,k) \in \vG'(\Bm) 
\text{ and } 
l \geq 1.
\]
Thus, 
$_R W(\la)$ 
is a highest weight module with a highest weight vector $\vf_{T^\la}$ of highest weight $\la$. 
Then, 
by Lemma \ref{Lemma universality of Weyl}, 
we have the surjective homomorphism 
\[ 
X_R : \, _R \D_n(\la) \ra \, _R W(\la) 
\text{ such that } 
1 \otimes v_\la \mapsto \vf_{T^\la}. 
\]
We should show that 
this homomorphism is an isomorphism, 
and 
it is enough to show the case where $R= \CA$ 
by the arguments of specializations. 
First, 
we consider the case where $R= \CK$. 
In this case, 
it is known that 
$_\CK \Sc_{n,r}$ is semi-simple, and 
$_\CK \D_n (\la)$ is irreducible 
(\cite[Theorem 2.16 (\roiv)]{W}). 
Thus, 
$X_\CK$ is injective.  
On the other hand, 
$X_\CK$ is obtained from $X_\CA$ by applying the right exact functor $\CK \otimes_{\CA} ?$, 
and 
the restriction of $X_{\CK}$ to $_\CA \D_n(\la)$ coincides with $X_{\CA}$.  
Thus, we have that $X_{\CA}$ is injective, 
hence,  
$X_{\CA}$ is an isomorphism.  
\end{proof}

\remark 
Lemma \ref{Lemma iso D W} 
was already proved in \cite{W} combined   with \cite[Theorem 5.16]{DR} implicitly. 
However, 
this identification is important in the later arguments, 
we gave the proof by using the universality.

\para 
Recall that $_R \Sc_{n,r}$ has the algebra anti-automorphism $\theta_n$, 
and we can consider the contravariant functor $\cas : \Sc_{n,r} \cmod \ra \Sc_{n,r} \cmod$ 
with respect to $\theta_n$. 
For $\la \in \vL_{n,r}^+$, 
put 
$_R \N_n(\la) = (_R \D_n(\la))^\cas$. 
Then 
$\{\,_R \D_n(\la)\,|\, \la \in \vL_{n,r}^+\}$ (resp. $\{\, _R \N_n(\la) \,|\, \la \in \vL_{n,r}^+\}$ 
gives a set of standard modules (resp. a set of costandard modules) of $_R \Sc_{n,r}$ 
in terms of quasi-hereditary algebras when $R$ is a field. 

When $R$ is a field, 
Let $_R \Sc_{n,r} \cmod^\D$ (resp. $_R \Sc_{n,r} \cmod^\N$) 
be the full subcategory of $_R \Sc_{n,r} \cmod$ 
consisting of  modules which have a filtration such that its successive quotients are isomorphic 
to standard modules (resp. costandard modules).



\section{Induction and restriction functors.} 

In this section, 
we give an injective homomorphism of algebras 
from a cyclotomic $q$-Schur algebra of lank $n$ 
to one of lank $n+1$. 
By using this embedding,  
we define induction and restriction functors 
between module categories of these two algebras.  

\para 
From now on, 
throughout this paper,  
we argue under the following setting: 
\begin{align}
\label{setting}
\begin{split}
&
\Bm = (m_1,\dots, m_r) \text{ such that } m_k \geq n+1 \text{ for all } k=1, \dots, r,  
\\
&
\Bm' = (m_1,\dots, m_{r-1}, m_r -1)
\\
&
_R \Sc_{n+1,r} = \, _R \Sc_{n+1,r}(\vL_{n+1,r}(\Bm)), 
\\
&
 _R \Sc_{n,r} = \, _R \Sc_{n,r} (\vL_{n,r} (\Bm')).
\end{split}
\end{align}
We will omit the subscript $R$ when there is no risk to confuse.

\remark 
The choice of $\Bm$ and $\Bm'$ in \eqref{setting}  
is essential when we consider an embedding from $\Sc_{n,r}$ to $\Sc_{n+1,r}$,  
and when we define induction and restriction functors by using the embedding.  
However, 
we remark that 
the choice of $\Bm$ and $\Bm'$ in \eqref{setting} 
is not essential up to Morita equivalent  when we study the representations of $\Sc_{n,r}$ and $\Sc_{n+1,r}$ 
if $R$ is a field (see Remark \ref{Remark Morita equivalent}).

\para 
\label{def inj weight}
We define the injective map 
\[
\g : \vL_{n,r}(\Bm') \ra \vL_{n+1,r}(\Bm), 
\quad 
(\la^{(1)}, \dots, \la^{(r-1)}, \la^{(r)}) \mapsto (\la^{(1)}, \dots, \la^{(r-1)}, \wh{\la}^{(r)}), 
\] 
where 
$\wh{\la}^{(r)}=(\la_1^{(r)}, \dots, \la_{m_r-1}^{(r)}, 1)$. 
Put 
$\vL_{n+1,r}^\g (\Bm) = \Im \g$, 
and we have 
\[ 
\vL^\g_{n+1,r}(\Bm) = \{ \mu =(\mu^{(1)}, \dots, \mu^{(r)}) \in \vL_{n+1,r}(\Bm) \,|\, \mu_{m_r}^{(r)}=1\}. 
\]

For $\la \in \vL_{n+1,r}^+$ and $\Ft \in \Std(\la)$, 
let 
$\Ft \setminus (n+1)$ 
be the standard tableau obtained by removing the node $x$ such that $\Ft(x) =n+1$, 
and 
denote the shape of  $\Ft \setminus (n+1)$ by $|\Ft \setminus (n+1)|$. 
Note that 
$x$ (in the last sentence) is a removable node of $\la$, 
and that  
$|\Ft \setminus (n+1)|= \la \setminus x$.

For 
$\la \in \vL_{n+1,r}^+$, $\mu \in \vL_{n+1,r}^\g(\Bm)$ and $T \in \CT_0(\la,\mu)$, 
let 
$T \setminus (m_r,r)$ be the tableau 
obtained by removing the node $x$ such that 
$T(x) =(m_r,r)$, 
and  denote 
the shape of $T \setminus (m_r,r)$ by $|T \setminus (m_r,r)|$. 
Note that 
$x$ (in the last sentence) 
is a removable node of $\la$, 
and that 
$|T \setminus (m_r,r)| = \la \setminus x$. 
It is clear that 
$T \setminus (m_r,r) \in \CT_0(\la \setminus x, \g^{-1}(\mu))$.

For $\la \in \vL_{n+1,r}^+$ and a removable node $x $ of $ \la$, 
we define the semi-standard tableau $T_x^\la \in \CT_0(\la)$ by 
\begin{align}
\label{definition Txla}
T_x^\la (a,b,c) = 
	\begin{cases} 
	(a,c) & \text{ if } (a,b,c) \not=x, 
	\\
	(m_r,r) & \text{ if } (a,b,c) = x. 
	\end{cases} 
\end{align}
From the definitions, 
we see that 
$T_x^\la \in \bigcup_{\mu \in \vL^\g_{n+1,r}(\Bm) } \CT_0(\la,\mu)$, 
and that 
$T_x^\la \setminus (m_r,r) =\nobreak T^{\la \setminus x}$.

\para 
Let 
$\CK \lan \bx \ran$ (resp. $\CK \lan \bx' \ran $) 
be the non-commutative polynomial ring over $\CK$ 
with indeterminate variables $\bx = \{x_{(i,k)} \,|\, (i,k) \in \vG'(\Bm)\}$ 
(resp. $\bx' = \{x_{(i,k)} \,|\, (i,k) \in\nobreak \vG'(\Bm')\}$). 
Note that $\vG'(\Bm') = \vG'(\Bm) \setminus \{ (m_{r} -1, r)\}$, 
and we have the natural injective map 
$\CK \lan \Bm' \ran \ra \CK \lan \Bm \ran$ 
such that 
$x_{(i,k)} \mapsto x_{(i,k)}$ for $(i,k) \in \vG'(\Bm')$. 
Under this injection, 
we regard a polynomial in $\CK \lan \Bm' \ran$ as a polynomial in $\CK \lan \Bm \ran$. 
It is similar for $\CK \lan \by \ran$ and $\CK \lan \by' \ran$. 

For the polynomials 
$g^\la_{(i,k)} (\bx', \by')  \in \CK\lan \bx' \ran \otimes_{\CK} \CK\lan \by' \ran$ 
($\la \in \vL_{n,r}(\Bm')$, $(i,k) \in \vG(\Bm')$) 
such as in Lemma \ref{Lemma polynomial JM}, 
we have the following lemma.

\begin{lem}
\label{Lemma relation JM}
Let 
$\la \in \vL_{n,r}(\Bm')$ and $(i,k) \in \vG'(\Bm')$. 
For 
$g^\la_{(i,k)} (\bx', \by')  \in \CK\lan \bx' \ran \otimes_{\CK} \CK\lan \by' \ran$ 
such that 
$g^\la_{(i,k)}(F,E) 1_\la = \s_{(i,k)}^\la$ in $_\CK \Sc_{n,r}$, 
we have 
\[ 
g^\la_{(i,k)} (F, E) 1_{\g(\la)} = \s^{\g(\la)}_{(i,k)} 
\]
in $_{\CK} \Sc_{n+1,r}$. 
In particular, 
we can take 
$g^\la_{(i,k)} (\bx', \by')$ 
as a polynomial 
$g^{\g(\la)}_{(i,k)}(\bx, \by)$ satisfying \eqref{polynomial JM} in $_\CK \Sc_{n+1,r}$.
\end{lem} 

\begin{proof}
Let 
$\iota : \, _\CK \He_{n,r} \ra \, _\CK \He_{n+1,r}$ 
be the natural injective homomorphism 
defined by 
$T_i \mapsto T_i$ ($0 \leq i \leq n-1$). 
Then we have 
$\iota(L_j) = L_j$ for $j = 1, \dots, n$. 

By \eqref{description E(i,k)} and  \eqref{description F(i,k)} (see also \cite[Lemma 6.10]{W}) 
we can check that  
\[ 
\iota (E_{(i,k)} (m_\la)) = E_{(i,k)} (m_{\g (\la)}), 
\quad 
\iota(F_{(i,k)} (m_\la)) = F_{(i,k)} (m_{\g(\la)}).
\]
These imply that, 
for $g (\bx', \by') \in \CK \lan \bx' \ran \otimes_{\CK} \CK \lan \by' \ran$,  
we have 
\[
\iota(g(F,E)(m_\la)) = g(F,E) (m_{\g(\la)})
\]
in $_\CK \He_{n+1,r}$. 
On the other hand, 
from the definition of $\s_{(i,k)}^\la$, 
we have 
\[ 
\iota( \s^\la_{(i,k)} (m_\la)) = \s^{\g(\la)}_{(i,k)} (m_{\g(\la)}).
\]
Then we have 
\begin{align*}
g^\la_{(i,k)} (F,E) 1_{\la} = \s_{(i,k)}^\la 
&\Leftrightarrow 
g^\la_{(i,k)} (F,E) (m_\la)  = \s_{(i,k)}^\la (m_\la) 
\\
&\Leftrightarrow 
\iota(g^\la_{(i,k)} (F,E) (m_\la))  = \iota(\s_{(i,k)}^\la (m_\la))
\\
&\Leftrightarrow 
g^\la_{(i,k)} (F,E) (m_{\g(\la)})  = \s_{(i,k)}^{\g(\la)} (m_{\g(\la)}) 
\\
&\Leftrightarrow 
g^\la_{(i,k)} (F,E) 1_{\g(\la)} = \s_{(i,k)}^{\g(\la)}.
\end{align*}
\end{proof}

Now, we can define the injective homomorphism from $\Sc_{n,r}$ to $\Sc_{n+1,r}$ 
as the following proposition.

\begin{prop}
\label{Prop embedding Scn to Scn+1}
There exists the algebra homomorphism 
$\iota : \Sc_{n,r} \ra \Sc_{n+1,r}$ 
such that 
\begin{align}
\label{def inj hom Scn to Scn+1}
E_{(i,k)}^{(l)} \mapsto E_{(i,k)}^{(l)} \xi, 
\quad 
F_{(i,k)}^{(l)} \mapsto F_{(i,k)}^{(l)} \xi,  
\quad 
1_\la \mapsto 1_{\g(\la)} 
\end{align}
for 
$ (i,k) \in \vG'(\Bm'), \, l \geq 1, \, \la \in \vL_{n,r}(\Bm')$, 
where 
$\xi = \sum_{\la \in \vL^\g_{n+1,r}(\Bm)} 1_\la$ is an idempotent of $\Sc_{n+1,r}$. 
In particular, 
we have that $\iota (1_{\Sc_{n,r}})=\xi$, 
and that $\iota(\Sc_{n,r}) \subsetneqq \xi \Sc_{n+1,r} \xi$, 
where 
$1_{\Sc_{n,r}}$ is the unit element of $\Sc_{n,r}$. 
Moreover, 
$\iota$ is injective. 
\end{prop}

\begin{proof} 
If $\iota$ is well-defined injective homomorphism given by \eqref{def inj hom Scn to Scn+1}, 
we easily see that 
$\iota (1_{\Sc_{n,r}})=\xi$, 
and that $\iota(\Sc_{n,r}) \subsetneqq \xi \Sc_{n+1,r} \xi$. 
Hence, 
it is enough to show the well-definedness and injectivity of $\iota$.  
 
First, we prove the statements for the algebras over $\CK$. 
In order to see the well-definedness 
of the homomorphism $\iota_{\CK} : \, _\CK \Sc_{n,r} \ra \, _\CK \Sc_{n+1,r} $ 
defined by \eqref{def inj hom Scn to Scn+1}, 
we should check the relations \eqref{S-1} - \eqref{S-8}. 
For the relations except \eqref{S-6}, 
it is clear, 
and we can check the relation \eqref{S-6} 
by Lemma \ref{Lemma relation JM}. 

We show that $\dim_{\CK} \, _\CK \Sc_{n,r} = \dim_{\CK} \iota_{\CK}(\, _\CK \Sc_{n,r})$,  
then this equality implies that $\iota_{\CK}$ is injective. 

For $\la \in \vL_{n,r}^+$, 
let 
$x$ be the addable node of $\la$ 
such that 
$x$ is minimum for the order $\succeq$ in the set of all addable nodes of $\la$, 
and put 
$\wh{\la} = \la \cup x$. 
Thus, we have $\wh{\la} \in \vL_{n+1,r}^+$. 
Note that 
$x$ is a removable node of $\wh{\la}$, 
we can take the semi-standard tableau 
$T_x^{\wh{\la}} $ defined by \eqref{definition Txla}. 
From the definitions, 
we see that 
$T_x^{\wh{\la}} \in \CT_0(\wh{\la}, \g(\la))$. 
When we regard $_\CK \D_{n+1}(\wh\la)$ 
as an $_\CK \Sc_{n,r}$-module through the homomorphism $\iota_{\CK}$, 
we see that 
$\vf_{T_x^{\wh{\la}}}$ is a weight vector of weight $\la$ since 
$\iota_{\CK}(1_\la) = 1_{\g(\la)}$ and $T_x^{\wh{\la}} \in \CT_0(\wh{\la}, \g(\la))$. 
On the other hand, 
for  $(i,k) \in \vG'(\Bm')$, 
we have $\g(\la) + \a_{(i,k)} \not\leq \wh{\la}$
since 
$x$ is minimum in the set of all addable nodes of $\la$. 
Thus, we have $\CT_0(\wh{\la}, \g(\la) + \a_{(i,k)}) = \emptyset$. 
This implies that 
$E_{(i,k)} \cdot \vf_{T_x^{\wh{\la}}} =0$ 
since 
$E_{(i,k)} \cdot \vf_{T_x^{\wh{\la}}}= 1_{\g(\la) + \a_{(i,k)}} E_{(i,k)} \cdot \vf_{T_x^{\wh{\la}}}$ 
together with \eqref{weight celluar basis of Sc}, 
where we consider the actions of $_\CK \Sc_{n+1,r}$. 
As a consequence, 
we see that 
$E_{(i,k)} \cdot \vf_{T_x^{\wh{\la}}} =0$ for any $(i,k) \in \vG'(\Bm')$, 
where we consider the action of $_\CK \Sc_{n,r}$ through $\iota_{\CK}$. 
This means that 
$\vf_{T_x^{\wh{\la}}}$ 
is a highest weight vector of weight $\la$, 
and 
$_\CK \Sc_{n,r}$-submodule 
 of $_\CK \D_{n+1}(\wh{\la})$ generated by $\vf_{T_x^{\wh{\la}}}$ 
is a highest weight module of highest weight $\la$. 
Thus, 
the universality of Weyl modules (Lemma \ref{Lemma universality of Weyl}) 
implies the isomorphism 
\[
_\CK \D_{n} (\la)   \cong \, _\CK \Sc_{n,r} \cdot \vf_{T_x^{\wh{\la}}}
\text{ as $_\CK \Sc_{n,r}$-modules}
\] 
since $_\CK \D_n(\la)$ is a simple $_\CK \Sc_{n,r}$-module. 
Now we proved that, 
for each $\la \in \vL_{n,r}^+$, 
the Weyl module 
$_\CK \D_n(\la)$ appears in $_\CK \D_{n+1}(\wh{\la})$ as an $_\CK \Sc_{n,r}$-submodule 
through $\iota_{\CK}$. 
Note that $_\CK \Sc_{n,r}$ is split semi-simple, 
the above arguments combined with Wedderburn's theorem implies that 
\[ 
\dim_{\CK} \, _\CK \Sc_{n,r} 
= \sum_{\la \in \vL_{n,r}^+}  (\dim_{\CK} \, _\CK \D_n(\la))^2 
= \dim_{\CK} \iota_{\CK}( \, _\CK \Sc_{n,r}).
\]
Now we proved that $\iota_{\CK}$ is injective. 

By restricting $\iota_{\CK}$ to $_\CA \Sc_{n,r}$, 
we have the injective homomorphism 
$\iota_{\CA} : \, _\CA \Sc_{n,r} \ra \, _\CA \Sc_{n+1,r}$ 
satisfying \eqref{def inj hom Scn to Scn+1}. 
In particular, we have 
$\iota_{\CA}( \, _\CA \Sc_{n,r}) \subset \xi \, _\CA \Sc_{n+1,r} \xi$. 
Put 
\begin{align*}
&M(\xi)= \{ \xi x 1_\la y \xi 
		\,|\, x \in \, _\CA \Sc_{n+1,r}^- , \, y \in \, _\CA \Sc_{n+1,r}^+,\, \la \in \vL^\g_{n+1,r}(\Bm) \},
\\
&M(\not\leq \xi) = \{ \xi x 1_\mu y \xi  
	\,|\, x \in \, _\CA \Sc_{n+1,r}^- , \, y \in \, _\CA \Sc_{n+1,r}^+,\, \mu \in \vL_{n+1,r}(\Bm) 
	\\
	&\hspace{10em}
	\text{ such that } \mu \not\leq \la \text{ for any } \la \in \vL_{n+1,r}^\g(\Bm) \}.
\end{align*}
By \eqref{def inj hom Scn to Scn+1} and the triangular decomposition of $_\CA \Sc_{n,r}$, 
we see that 
\begin{align}
\label{image Scn}
\iota_{\CA} (\, _\CA\Sc_{n,r}) = M(\xi).
\end{align} 
Moreover, 
we claim that 
\begin{align}
\label{decom xi Sc xi}
\xi \, _\CA \Sc_{n+1,r} \xi = M(\xi) \oplus M( \not\leq \xi).
\end{align}
Thanks to the triangular decomposition of $_\CA \Sc_{n+1,r}$, 
we have 
\[ 
\xi \, _\CA \Sc_{n+1,r} \xi 
	= \{ \xi x 1_\la y \xi  
		\,|\, x \in \, _\CA \Sc_{n+1,r}^- , \, y \in \, _\CA \Sc_{n+1,r}^+,\, \la \in \vL_{n+1,r}(\Bm)\}.
\]
Thus, in order to show \eqref{decom xi Sc xi}, it is enough to show that 
\begin{align}
\label{xi x mu y xi =0}
\xi x 1_\mu y \xi =0 
\text{ if } 
\mu < \la \text{ for some } \la \in \vL_{n+1,r}^\g(\Bm) \text{ and } 
		\mu \in \vL_{n+1,r}(\Bm) \setminus \vL_{n+1,r}^\g(\Bm).
\end{align}
From the definitions, 
for $\mu \in \vL_{n+1,r}(\Bm) \setminus \vL_{n+1,r}^\g(\Bm) $, 
we see that 
$ \mu_{m_r}^{(r)} \geq 2 $ 
if $\mu < \la$ for some $\vL_{n+1,r}^\g(\Bm)$, 
and 
we see that 
$\xi x 1_\mu y \xi=0$ ($ x \in \, _\CA \Sc_{n+1,r}^- , \, y \in \, _\CA \Sc_{n+1,r}^+$) 
if $\mu_{m_r}^{(r)} \geq 2$ from the relations \eqref{S-1}-\eqref{S-8}. 
These imply \eqref{xi x mu y xi =0}, and we have \eqref{decom xi Sc xi}. 
Put $\xi'=\sum_{\la \in \vL_{n+1,r}(\Bm) \setminus \vL_{n+1,r}^\g(\Bm)} 1_\la$, 
we have 
$1_{_\CA \Sc_{n+1,r}} = \xi + \xi'$. 
Then, by   \eqref{decom xi Sc xi}, we have   
\begin{align*}
_\CA \Sc_{n+1,r} = M(\xi) \oplus M(\not\leq \xi) 
	\oplus \xi \, _\CA \Sc_{n+1,r} \xi' \oplus \xi' \, _\CA \Sc_{n+1,r} \xi\oplus \xi' \, _\CA \Sc_{n+1,r} \xi'.
\end{align*} 
Combining with \eqref{image Scn}, 
we see that  
the injective homomorphism 
$ \iota_{\CA} : \, _\CA \Sc_{n,r} \ra\nobreak \, _\CA \Sc_{n+1,r}$ 
is split as a $\CA$-homomorphism. 
Thus, by the specialization of $\iota_{\CA}$ to $R$, 
we have the injective homomorphism 
$\iota_R : \, _R \Sc_{n,r} \ra \, _R \Sc_{n+1,r} $  
satisfying \eqref{def inj hom Scn to Scn+1}. 
\end{proof}

By \eqref{def inj hom Scn to Scn+1} and Lemma \ref{Lemma image theta}, 
we have the following corollary. 

\begin{cor}
When we regard $\Sc_{n,r}$ as a subalgebra of $\Sc_{n+1,r}$ through the injective homomorphism 
$\iota : \Sc_{n,r} \ra \Sc_{n+1,r}$, 
the anti-involution $\theta_n$ on $\Sc_{n,r}$ 
coincides with the restriction of the anti-involution $\theta_{n+1}$ on $\Sc_{n+1,r}$.  
\end{cor}

\para 
From now on, 
we regard $\Sc_{n,r}$ as a subalgebra of $\Sc_{n+1,r}$ through the injective homomorphism 
$\iota : \Sc_{n,r} \ra \Sc_{n+1,r}$. 
As defined in Proposition\ref{Prop embedding Scn to Scn+1}, 
put 
$\xi = \sum_{\la \in \vL_{n+1,r}^\g(\Bm)} 1_\la$. 
Since $\xi$ is an idempotent of $\Sc_{n+1,r}$, 
we see that $\xi \Sc_{n+1,r} \xi$ is a subalgebra of $\Sc_{n+1,r}$ with the unit element $\xi$. 
Thus, 
$\xi \Sc_{n+1,r}$ (resp. $\Sc_{n+1,r} \xi$) 
is an $(\xi \Sc_{n+1,r} \xi, \Sc_{n+1,r})$-bimodule (resp. $(\Sc_{n+1,r}, \xi \Sc_{n+1,r} \xi)$-bimodule) 
by the multiplications. 
Note that $\iota(\Sc_{n,r}) \subset \xi \Sc_{n+1,r} \xi$ and $\iota(1_{\Sc_{n,r}}) =\xi$, 
we can restrict the action of $\xi \Sc_{n+1,r} \xi$ to $\Sc_{n,r}$ through $\iota$. 
Thus, 
$\xi \Sc_{n+1,r}$ (resp. $\Sc_{n+1,r} \xi$) 
turns out to be an $(\Sc_{n,r}, \Sc_{n+1,r})$-bimodule (resp. $(\Sc_{n+1,r}, \Sc_{n,r})$-bimodule) 
by restriction. 
We define a restriction functor $\Res^{n+1}_n : \Sc_{n+1,r} \cmod \ra \Sc_{n,r} \cmod$ by 
\[ 
\Res^{n+1}_n = \Hom_{\Sc_{n+1,r}} (\Sc_{n+1,r} \xi , ? ) \cong \xi \Sc_{n+1,r} \otimes_{\Sc_{n+1,r}} ?. 
\]
We also define two induction functors 
$\Ind^{n+1}_n, \coInd^{n+1}_n : \Sc_{n,r} \cmod \ra \Sc_{n+1,r} \cmod$ 
by 
\begin{align*}
&\Ind^{n+1}_n = \Sc_{n+1,r} \xi \otimes_{\Sc_{n,r}} ?,
\\
& \coInd^{n+1}_n = \Hom_{\Sc_{n,r}} ( \xi \Sc_{n+1,r} , ?).
\end{align*}
By the definition, we have the following. 
\begin{itemize}
\item 
$\Res^{n+1}_n$ is exact, 
$\Ind^{n+1}_n$ is right exact 
and 
$\coInd^{n+1}_n$ is left exact. 

\item 
$\Ind^{n+1}_n$ is left adjoint to $\Res^{n+1}_n$. 

\item 
$\coInd^{n+1}_n$ is right adjoint to $\Res^{n+1}_n$. 
\end{itemize}

We have the following commutativity with these functors and  the contravariant functors 
$\cas$ with respect to the anti-involution $\theta_n$ and $\theta_{n+1}$.

\begin{lem}
\label{Lemma commute Res cas, Ind cas}
We have th following isomorphisms of functors. 
\begin{enumerate}
\item 
$\cas \circ \Res^{n+1}_n \cong \Res^{n+1}_n \circ \cas$. 

\item 
$\cas \circ \Ind^{n+1}_n \cong \coInd^{n+1}_n \circ \cas$.
\end{enumerate}
\end{lem}

\begin{proof}
We claim that 
\begin{align}
\label{iso Hom dual 1}
&\Hom_R (\Sc_{n+1,r} \xi, R) \cong (\xi \Sc_{n+1,r})^\cas 
	\quad \text{ as $(\Sc_{n,r}, \Sc_{n+1,r})$-bimodules}, 
\\
\label{iso Hom dual 2}
&\Hom_R (\xi \Sc_{n+1,r}, R) \cong (\Sc_{n+1,r} \xi)^\cas 
	\quad \text{ as $(\Sc_{n+1,r}, \Sc_{n,r})$-bimodules}, 
\end{align}
where 
the $(\Sc_{n,r}, \Sc_{n+1,r})$-bimodule $(\xi \Sc_{n+1,r})^\cas $ 
is obtained by twisting  the action 
of $(\Sc_{n+1,r}, \Sc_{n,r})$-bimodule $\Hom_{R}(\xi \Sc_{n+1,r}, R)$ 
via the anti-involutions  $\theta_n$ and $\theta_{n+1}$. 
It is similar for the $(\Sc_{n+1,r}, \Sc_{n,r})$-bimodule $(\Sc_{n+1,r} \xi)^\cas$. 
We prove  only \eqref{iso Hom dual 1} since we can prove \eqref{iso Hom dual 2} in a similar way. 
We define a map 
\[
G : \Hom_{R} (\Sc_{n+1,r} \xi, R) \ra (\xi \Sc_{n+1,r})^\cas
\] 
by 
$G(\vf)(\xi s)= \vf  \big( \theta_{n+1}(s) \xi \big)$ 
for $\vf \in \Hom_R(\Sc_{n+1,r} \xi, R)$ and $s \in \Sc_{n+1,r}$.  
We also define a map 
\[
H : (\xi \Sc_{n+1,r})^\cas \ra \Hom_R(\Sc_{n+1,r} \xi, R)
\] 
by 
$H(\psi)(s \xi) = \psi ( \xi \theta_{n+1}(s))$ 
for 
$\psi \in (\xi \Sc_{n+1,r})^\cas$ and $s \in \Sc_{n+1,r}$. 
Then we can check that 
$G$ and $H$ are well-defined  $(\Sc_{n,r}, \Sc_{n+1,r})$-bimodule homomorphisms, 
and they give an inverse map for each other. 
Thus, we have \eqref{iso Hom dual 1}. 

For $M \in \Sc_{n+1,r} \cmod$, 
we have the following natural isomorphisms 
\begin{align*}
\cas \circ \Res^{n+1}_n (M) 
&= 
\Hom_R (\xi \Sc_{n+1,r} \otimes_{\Sc_{n+1,r}} M, R) 
\\
&
\cong 
\Hom_{\Sc_{n+1,r}} (M, \Hom_R (\xi \Sc_{n+1,r}, R))
\\
&
\cong 
\Hom_{\Sc_{n+1,r}} (M, (\Sc_{n+1,r} \xi )^\cas ) \quad (\text{because of } \eqref{iso Hom dual 2}) 
\\
&
\cong 
\Hom_{\Sc_{n+1,r}} (\Sc_{n+1,r}\xi,  M^\cas)
\\
&= 
\Res^{n+1}_n \circ \cas (M), 
\end{align*}
and this implies (\roi).

For $N \in \Sc_{n,r}\cmod$, we have the following natural isomorphisms 
\begin{align*}
\cas \circ \Ind^{n+1}_n (N) 
&= 
\Hom_R( \Sc_{n+1,r} \xi \otimes_{\Sc_{n,r}} N , R) 
\\
&\cong 
\Hom_{\Sc_{n,r}} (N, \Hom_R(\Sc_{n+1,r} \xi , R))
\\
&\cong 
\Hom_{\Sc_{n,r}} (N, (\xi \Sc_{n+1,r})^\cas ) \quad (\text{because of } \eqref{iso Hom dual 1}) 
\\
&\cong 
\Hom_{\Sc_{n,r}} (\xi \Sc_{n+1,r}, N^\cas) 
\\
&= 
\coInd^{n+1}_n \circ \cas (N),
\end{align*}
and this implies (\roii). 
\end{proof}

The last of this section, 
we prepare the following lemma for later arguments. 

\begin{lem}
\label{Lemma Ind D Ind N in K0}
Assume that  $R$ is a field.  
For $\la \in \vL_{n,r}^+$, 
we have 
\begin{align*}
[ \Ind^{n+1}_n( \D_n(\la)) ]
 =
[ \Ind^{n+1}_n (\N_n(\la)) ] 
\quad 
\text{ in } K_0(\Sc_{n+1,r} \cmod). 
\end{align*}
\end{lem}
\begin{proof}
Let $(K,\wh{R}, R)$ be a suitable modular system, 
namely $\wh{R}$ is a discrete valuation ring such that $R$ is the residue field of $\wh{R}$, 
and $K$ is the quotient field of $\wh{R}$ such that $_K \Sc_{n,r}$ is semi-simple 
(e.g. see \cite[section 5.3]{Mat04}). 
Let $X$ be one of $K$, $\wh{R}$ or $R$. 
For $\la \in \vL_{n,r}^+$, 
we see that 
$_{X} \Sc_{n+1,r} \xi \otimes_{_X \Sc_{n,r}} \, _X \D_n(\la)$ 
is generated by $\{ \xi \otimes \vf_T \,|\, T \in \CT_0(\la)\}$ as $_X \Sc_{n+1,r}$-modules,  
and that 
$\xi \otimes \vf_T \not=0$ for any $T \in \CT_0(\la)$ 
since $\xi$ is regarded as the identity element of $_X \Sc_{n,r}$. 
Thus, 
$_{\wh{R}} \Sc_{n+1,r} \xi \otimes_{_{\wh{R}} \Sc_{n,r}} \, _{\wh{R}} \D_n(\la)$ 
is an $_{\wh{R}} \Sc_{n+1,r}$-submodule of 
$_{_K} \Sc_{n+1,r} \xi \otimes_{_K \Sc_{n,r}} \, _{K} \D_n(\la)$ 
generated by $\{ \xi \otimes \vf_T \,|\, T \in \CT_0(\la)\}$. 
In particular, 
$_{\wh{R}} \Sc_{n+1,r} \xi \otimes_{_{\wh{R}} \Sc_{n,r}} \, _{\wh{R}} \D_n(\la)$  
is torsion free, 
thus it is a full rank  $\wh{R}$-lattice of 
$_{K} \Sc_{n+1,r} \xi \otimes_{_K \Sc_{n,r}} \, _{K} \D_n(\la)$. 
Similarly, 
$_{\wh{R}} \Sc_{n+1,r} \xi \otimes_{_{\wh{R}} \Sc_{n,r}} \, _{\wh{R}} \N_n(\la)$  
is a full rank  $\wh{R}$-lattice of 
$_{K} \Sc_{n+1,r} \xi \otimes_{_K \Sc_{n,r}} \, _{K} \N_n(\la)$. 
Moreover, 
by the general theory of cellular algebras,  
we have $_K \D_n(\la) \cong \, _K\N_n (\la)$ 
since $_K \Sc_{n,r}$ is semi-simple. 
Then, 
the decomposition map implies 
\begin{align*}
[ _{R} \Sc_{n+1,r} \xi \otimes_{_R \Sc_{n,r}} \, _{R} \D_n(\la)]
 =
[ _{R} \Sc_{n+1,r} \xi \otimes_{_R \Sc_{n,r}} \, _{R} \N_n(\la) ] 
\quad 
\text{ in } K_0(\, _R \Sc_{n+1,r} \cmod). 
\end{align*}
\end{proof}


\section{Restricted and induced Weyl modules}   

In this section, 
we describe filtrations of restricted and induced Weyl modules 
(resp. costandard modules) 
whose successive quotients are isomorphic to Weyl modules 
(resp. costandard modules). 

\para 
It is known that 
there exists the injective  homomorphism of algebras 
\begin{align}
\label{embedding H}
\iota^{\He} : \He_{n,r} \ra \He_{n+1,r}   
\text{ such that } 
T_i \mapsto T_i 
\text{ for } 
i=0,1,\dots, n-1.
\end{align} 
We regard $\He_{n,r}$ as a subalgebra of $\He_{n+1,r}$ through $\iota^\He$. 
\para 
We recall that, 
for $\la \in \vL_{n+1,r}^+$, 
the Weyl module 
$\D_{n+1}(\la)$ of $\Sc_{n+1,r}$ has an $R$-free basis 
$\{\vf_T \,|\, T \in \CT_0(\la)\}$. 
From the definition, 
we have that 
\[
\Res^{n+1}_n (\D_{n+1,r}(\la)) = \xi \cdot \D_{n+1,r}(\la).
\] 
Thus, we see that 
$\Res^{n+1}_n (\D_{n+1,r}(\la))$ has an $R$-free basis 
\[
\{ \vf_T \,|\, T \in \CT_0(\la,\mu) \text{ for some } \mu \in \vL_{n+1,r}^\g(\Bm)\}
\] 
thanks to \eqref{weight celluar basis of Sc}.


\begin{prop}
\label{Prop action E F vfT}
Let 
$\la \in \vL_{n+1,r}^+$, 
$\mu \in \vL_{n+1,r}^\g(\Bm)$, 
$T \in \CT_0(\la,\mu)$. 
For $(i,k) \in \vG'(\Bm')$, we have the following. 
\begin{enumerate}
\item 
$\dis E_{(i,k)} \cdot \vf_{T} 
	= \sum_{S \in \CT_0(\la, \mu + \a_{(i,k)}) \atop |S \setminus (m_r,r)| \geq  |T \setminus (m_r,r)|} 
	r_S \vf_S \quad (r_S \in R)$.
	
\item 
$\dis F_{(i,k)} \cdot \vf_{T} 
	= \sum_{S \in \CT_0(\la, \mu - \a_{(i,k)}) \atop |S \setminus (m_r,r)| \geq  |T \setminus (m_r,r)|} 
	r_S \vf_S \quad (r_S \in R)$.
	
\end{enumerate}
\end{prop}

\begin{proof}
We prove only (\roi) since we can prove (\roii) in a similar way. 

If $\mu + \a_{(i,k)} \not\in \vL_{n+1,r}(\Bm)$, 
we have 
$E_{(i,k)} \cdot \vf_T = E_{(i,k)} 1_\mu \cdot \vf_T = 0$ by \eqref{S-2}, 
and there is nothing to prove. 
Thus, we assume that $\mu + \a_{(i,k)} \in \vL_{n+1,r}(\Bm)$. 
Then we have 
$E_{(i,k)} \cdot \vf_T = 1_{\mu + \a_{(i,k)}} E_{(i,k)} \cdot \vf_T$, 
and this implies that 
$S \in \CT_0 (\la, \mu + \a_{(i,k)})$
if $r_S \not=0$ thanks to \eqref{weight celluar basis of Sc}. 
Hence, 
it is enough to prove that 
$|S \setminus (m_r,r)| \geq  |T \setminus (m_r,r)|$ if $r_S \not=0$.

By \cite[Proposition 6.3]{DJM98}, we can write 
$m_{TT^\la} = m_\mu h$ for some $h \in \He_{n+1,r}$. 
Then, by \eqref{description E(i,k)}, we have 
\begin{align*}
E_{(i,k)} \cdot \vf_{T T^\la} (m_\la) 
&= E_{(i,k)} (m_{T T^\la}) 
\\
&= E_{(i,k)} (m_\mu h)
\\
&=
\Big( q^{- \mu_{i+1}^{(k)} +1} 
		\big( \sum_{x \in X_\mu^{\mu+\a_{(i,k)}}} q^{\ell (x)} T_x^\ast \big) h^\mu_{+(i,k)} m_\mu 
\Big) \cdot h			
\\
&= 
q^{- \mu_{i+1}^{(k)} +1}  \big( \sum_{x \in X_\mu^{\mu+\a_{(i,k)}}} q^{\ell (x)} T_x^\ast \big) h^\mu_{+(i,k)} 
	m_{TT^\la}
\\
&= 
\sum_{\Ft\in \Std(\la) \atop \mu(\Ft)=T} 
	q^{- \mu_{i+1}^{(k)} +1}  \big( \sum_{x \in X_\mu^{\mu+\a_{(i,k)}}} q^{\ell (x)} T_x^\ast \big) h^\mu_{+(i,k)}
	m_{\Ft \Ft^\la}.
\end{align*}
(Note that $\Ft^\la$ is the unique standard tableau $\Ft \in \Std(\la)$ such that $\la(\Ft) =T^\la$.)
Since $\mu \in \vL^\g_{n+1,r}(\Bm)$ and  $\mu + \a_{(i,k)} \in \vL_{n+1,r}(\Bm)$, 
we have $\mu_{m_r}^{(r)}=1$ and $\mu_{i+1}^{(k)} \geq 1$. 
These imply 
$\sum_{l=1}^{k-1} |\mu^{(l)}| + \sum_{j=1}^i \mu_j^{(k)} \leq n-1$. 
Then, from the definitions of $X_\mu^{\mu+\a_{(i,k)}}$ and $h_{+(i,k)}^\mu$, 
we see that 
$q^{- \mu_{i+1}^{(k)} +1}  \big( \sum_{x \in X_\mu^{\mu+\a_{(i,k)}}} q^{\ell (x)} T_x^\ast \big) h^\mu_{+(i,k)} 
\in \He_{n,r}$, 
where we regard $\He_{n,r}$ as a subalgebra of $\He_{n+1,r}$ by \eqref{embedding H}. 
Thus, by \cite[Proof of Proposition 1.9]{AM}, 
we have 
\begin{align}
\label{E vfT}
E_{(i,k)} \cdot \vf_{T T^\la} (m_\la)  
&= 
\sum_{\Ft\in \Std(\la) \atop \mu(\Ft)=T} 
	q^{- \mu_{i+1}^{(k)} +1}  \big( \sum_{x \in X_\mu^{\mu+\a_{(i,k)}}} q^{\ell (x)} T_x^\ast \big) h^\mu_{+(i,k)}
	m_{\Ft \Ft^\la} 
\\
& \equiv 
\sum_{\Ft \in \Std(\la) \atop \mu(\Ft)=T} 
\Big( \sum_{\Fs \in \Std(\la) \atop |\Fs \setminus n+1| \geq |\Ft \setminus n+1|} 
		r^{\Ft}_{ \Fs} m_{\Fs \Ft^\la}
\Big)  
\mod \He_{n+1,r} (>\la) 
\quad (r^\Ft_\Fs \in R),
\notag 
\end{align}
where 
$\He_{n+1,r}(>\la) $ is an $R$-submodule of $\He_{n+1,r}$ spanned by 
$\{ m_{\Fu \Fv} \,|\, \Fu, \Fv \in \Std(\la') \break \text{ for some } 
	\la' \in \vL_{n+1,r}^+ \text{ such that } \la'>\la \}$. 
Since $\mu_{m_r}^{(r)}=1$, 
we see that 
$|\Ft \setminus n+1|$ 
does not depend on a choice of $\Ft \in \Std(\la)$ such that $\mu(\Ft)=T$. 
Then, 
take and fix a standard tableau $\Ft' \in \Std(\la)$ such that $\mu (\Ft')=T$, 
and \eqref{E vfT} implies 
\begin{align}
\label{E vfT 2}
E_{(i,k)} \cdot \vf_{T T^\la} (m_\la)  
\equiv 
\sum_{\Fs \in \Std(\la) \atop |\Fs \setminus n+1| \geq |\Ft' \setminus n+1|} r_\Fs m_{\Fs \Ft^\la} 
\mod \He_{n+1,r}(>\la) \quad (r_\Fs \in R). 
\end{align}
On the other hand, 
by a general theory of cellular algebras together with \eqref{weight celluar basis of Sc}, 
we can write 
\begin{align*}
E_{(i,k)} \cdot \vf_{TT^\la} \equiv \sum_{S \in \CT_0(\la, \mu+ \a_{(i,k)})} r_S \vf_{ST^\la} 
\mod \Sc_{n+1,r}(>\la) \quad (r_S \in R).
\end{align*}
Thus, we have 

\begin{align}
\label{E vfT 3}
E_{(i,k)} \cdot \vf_{T T^\la} (m_\la) 
&\equiv \sum_{S \in \CT_0(\la,\mu+\a_{(i,k)})} r_{S} m_{S T^\la}  
\\
&=
\sum_{S \in \CT_0(\la,\mu + \a_{(i,k)})} r_S 
	\Big( \sum_{\Fs \in \Std(\la) \atop (\mu+\a_{(i,k)})(\Fs)=S} q^{\ell (d(\Fs))} m_{\Fs \Ft^\la} \Big) 
\mod \He_{n+1,r}(>\la).
\notag
\end{align}
Note that $\mu + \a_{(i,k)} \in \vL^\g_{n+1,r}(\Bm)$, 
we can easily check that 
$|S \setminus (m_r,r)|=|\Fs \setminus n+1|$ 
for $S \in \CT_0(\la,\mu + \a_{(i,k)})$ and $\Fs \in \Std(\la)$ such that $(\mu +\a_{(i,.k)})(\Fs)=S$. 
Similarly, we have $|T \setminus (m_r,r)|=|\Ft' \setminus n+1|$. 
Thus, by comparing the coefficients in \eqref{E vfT 2} and \eqref{E vfT 3}, 
we have 
$|S \setminus (m_r,r)| \geq  |T \setminus (m_r,r)|$ if $r_S \not=0$. 
\end{proof}

Now we can describe filtrations of restricted and induced Weyl modules 
(resp. costandard modules) as follows.

\begin{thm}
\label{Thm Res Ind Weyl}
Assume that $R$ is a field, 
we have the following. 
\begin{enumerate}
\item 
For $\la \in \vL_{n+1,r}^+$, 
there exists a filtration of $\Sc_{n,r}$-modules 
\[ 
\Res^{n+1}_n (\D_{n+1}(\la)) = M_1 \supset M_2 \supset \dots \supset M_k \supset M_{k+1}=0, 
\]
such that 
$M_i / M_{i+1} \cong \D_n (\la \setminus x_i)$, 
where 
$x_1,x_2, \dots, x_k$ are all removable nodes of $\la$ such that 
$x_1 \succ x_2 \succ \dots \succ x_k$. 

\item 
For $\la \in \vL_{n+1,r}^+$, 
there exists a filtration of $\Sc_{n,r}$-modules 
\[ 
\Res^{n+1}_n (\N_{n+1}(\la)) = N_k \supset N_{k-1} \supset \dots \supset N_1 \supset N_{0}=0, 
\]
such that 
$N_i / N_{i-1} \cong \N_n (\la \setminus x_i)$, 
where 
$x_1,x_2, \dots, x_k$ are all removable nodes of $\la$ such that 
$x_1 \succ x_2 \succ \dots \succ x_k$. 

\item 
For $\mu \in \vL_{n,r}^+$, there exists a filtration of $\Sc_{n+1,r}$-modules 
\[ 
\Ind^{n+1}_n (\D_n(\mu)) = M_1 \supset M_2 \supset \dots \supset M_k \supset M_{k+1}=0 
\]
such that $M_i / M_{i+1} \cong \D_{n+1}(\mu \cup x_i)$, 
where 
$x_1, x_2, \dots, x_k$ are all addable nodes of $\mu$ 
such that 
$x_k \succ x_{k-1} \succ \dots \succ x_1$. 

\item 
For $\mu \in \vL_{n,r}^+$, there exists a filtration of $\Sc_{n+1,r}$-modules 
\[ 
\coInd^{n+1}_n (\N_n(\mu)) = N_k \supset N_{k-1} \supset \dots \supset N_1 \supset N_{0}=0 
\]
such that $N_i / N_{i-1} \cong \N_{n+1}(\mu \cup x_i)$, 
where 
$x_1, x_2, \dots, x_k$ are all addable nodes of $\mu$ 
such that 
$x_k \succ x_{k-1} \succ \dots \succ x_1$. 
\end{enumerate}
\end{thm}

\begin{proof}
(\roi) 
Until declining, 
let $R$ be an arbitrary commutative ring. 
Put 
\[
\CT_0^\g(\la) = \bigcup_{\mu \in \vL^\g_{n+1,r}(\Bm)} \CT_0(\la,\mu).
\]  
Then, 
$\Res^{n+1}_n (\D_{n+1,r}(\la))$ has an $R$-free basis 
$\{\vf_T \,|\, T \in \CT_0^\g(\la)\}$. 
For $T \in \CT_0^\g (\la)$, 
there exists the unique removable node $x$ of $\la$ 
such that 
$T(x)=(m_r,r)$ since $\mu_{m_r}^{(r)}=1$.  
Let $x_1, x_2,\dots, x_k$ 
be all removable nodes of $\la$ 
such that 
$x_1 \succ x_2 \succ \dots \succ x_k$ 
(note that the order $\succeq $ determines a total order on the set of removable nodes of $\la$). 
Let $M_i$ 
be an $R$-submodule of $\Res^{n+1}_n (\D_{n+1}(\la))$ 
spanned by 
\[ 
\{ \vf_T \,|\, T \in \CT_0^\g(\la) \text{ such that } 
	T(x_j) =(m_r,r) \text{ for some } j \geq i
\}.
\]
Then we have a filtration of $R$-modules 
\begin{align}
\label{filtration Res D}
\Res^{n+1}_n (\D_{n+1}(\la)) = M_1 \supset M_2 \supset \dots \supset M_k \supset M_{k+1}=0. 
\end{align}

For $T \in \CT_0^{\g}(\la)$ such that $T(x_i)=(m_r,r)$, 
we have 
$|T \setminus (m_r,r)| = \la \setminus x_i$ by the definition. 
Note that 
$\la \setminus x_j > \la \setminus x_i$ if and only if $x_j \prec x_i$ (i.e. $j > i$). 
Then, 
for $S,T \in \CT_0^\g (\la)$ such that $S(x_j)=T(x_i)= (m_r,r)$, 
we have that 
\begin{align}
\label{S T}
|S \setminus (m_r,r)| > |T \setminus (m_r,r)| 
\text{ if and only if } 
j  >  i.
\end{align} 
By Proposition \ref{Prop action E F vfT} and \eqref{S T}, 
we see that 
$M_i$ is an $\Sc_{n,r}$-submodule of 
$\Res_n^{n+1}(\D_{n+1}(\la))$ 
for each $i=1, \dots, k$, 
and the filtration \eqref{filtration Res D} 
is a filtration of $\Sc_{n,r}$-modules.

From the definition, 
$M_i/M_{i+1}$ has an $R$-free basis 
\[
\{\vf_T + M_{i+1} \,|\, T \in \CT_0^\g(\la) \text{ such that } T(x_i)=(m_r,r)\}.
\]  
Note that 
$T_{x_i}^\la \in \CT_0^\g(\la)$ and $T_{x_i}^\la (x_i)=(m_r,r)$.
Let 
$x_i =(a,b,c)$, 
and put 
$\t = \la - (\a_{(a,c)} + \a_{(a+1,c)} + \dots + \a_{(m_r -1,r)})$. 
Then we have 
$T_{x_i}^\la \in \CT_0(\la,\t)$. 
Note that 
$E_{(j,l)} \cdot \vf_{T_{x_i}^\la}=1_{\t + \a_{(j,l)}} E_{(j,l)} \cdot \vf_{T_{x_i}^\la}$ 
is a linear combination of $\{\vf_T \,|\, T \in \CT_0(\la,\t + \a_{(j,l)}) \}$, 
and 
that $\CT_0(\la,\t + \a_{(j,l)}) = \emptyset$ unless $\la \geq \t + \a_{(j,l)}$. 

If $(j,l) \succ (a,c)$, 
we have $E_{(j,l)} \cdot \vf_{T_{x_i}^\la} =0$ since $\la \not\geq \t + \a_{(j,l)}$. 

Assume that $(j,l) \preceq (a,c)$, 
and we have 
\begin{align}
\label{sum la tau alpha}
\sum_{g=1}^{l-1}|\la^{(g)}| + \sum_{b=1}^{j} \la^{(l)}_b 
	= \sum_{g=1}^{l-1}|(\t + \a_{(j,l)})^{(g)}| + \sum_{b=1}^{j} (\t + \a_{(j,l)})^{(l)}_b. 
\end{align}
By \eqref{sum la tau alpha} together with the definition of  semi-standard tableaux, 
we can easily check that 
$S((a' , b' , c')) \leq (j,l)$ for any $S \in \CT_0(\la,\t + \a_{(i,k)})$ 
and any $(a', b', c') \in [\la]$ such that  $(a', c') \succeq (j,  l)$.  
This implies that 
\begin{align}
\label{S > Txla}
|S \setminus (m_r,r)| \not= |T_{x_i}^\la \setminus (m_r,r)|  
\text{ for any } 
S \in \CT_0(\la,\t + \a_{(j,l)})
\end{align} 
since 
$(a,c) \succeq (j, l)$ 
and 
$T^\la_{x_i}(x_i)=T^\la_{x_i}((a,b,c)) =(m_r,r) \geq (j,l)$.

Thus, 
Proposition \ref{Prop action E F vfT} (\roi)  together with  \eqref{S > Txla} implies 
\begin{align*}
\label{E vf T x la}
E_{(j,l)} \cdot \vf_{T_{x_i}^\la}  \in M_{i+1}  
\text{ for any }  (j,l) \in \vG'(\Bm').
\end{align*}
Thus, 
$\vf_{T_{x_i}^\la} + M_{i+1} \in M_{i}/M_{i+1}$ is a highest weight vector of  weight $\la \setminus x_i$ 
as an element of the $ \Sc_{n,r}$-module. 
Thus, 
by the universality of Weyl modules (Lemma \ref{Lemma universality of Weyl}), 
we have the  surjective $_R \Sc_{n,r}$-homomorphism 
\[ 
X_R :  \, _R \D_n(\la \setminus x_i) \ra  \, _R \Sc_{n,r} \cdot (\vf_{T_{x_i}^\la} + M_{i+1}). 
\]

Since $_\CK \D_n(\la \setminus x_i)$ is a simple $_\CK \Sc_{n,r}$-module, 
$X_{\CK}$ 
is  injective. 
Note that 
$X_R$ is determined by  $X_R (1 \otimes v_{\la \setminus x_i}) = \vf_{T_{x_i}^\la} + M_{i+1}$, 
we see that 
$X_\CA$ is the restriction of $X_{\CK}$ to $_\CA \D_n(\la \setminus x_i)$. 
Thus, 
$X_{\CA}$ is also injective, and 
$X_\CA$ is an isomorphism.  
Then, 
by the argument of specialization, 
we have 
$X_R$ is an isomorphism for an arbitrary commutative ring $R$. 

Assume that $R$ is a field. 
Since $\D_n(\la \setminus x_i) \cong \Sc_{n,r} \cdot  (\vf_{T_{x_i}^\la} + M_{i+1})$
 is an $\Sc_{n,r}$-submodule of $ M_i / M_{i+1}$, 
we have that 
$\D_n(\la \setminus x_i) \cong M_i/M_{i+1}$ 
by comparing the dimensions of $\D_n(\la \setminus x_i)$ and of $M_i/M_{i+1}$.
Now we proved (\roi).

(\roii) is obtained by applying the contravariant functor $\cas$ to (\roi) 
thanks to Lemma \ref{Lemma commute Res cas, Ind cas} (\roi). 

Next, we prove (\roiv). 
For $\la \in \vL_{n+1,r}^+$, 
let 
$\D_{n+1}^\sharp(\la)$ be an $R$-submodule of $\Sc_{n+1,r}/ \Sc_{n+1,r}(>\la)$ 
spanned by 
$\{ \vf_{T^\la T} + \Sc_{n+1,r}(>\la) \,|\, T \in \CT_0(\la)\}$.
Then, by a general theory of cellular algebras, 
it is known that 
$\D_{n+1}^\sharp (\la)$ is a right $\Sc_{n+1,r}$-submodule of $\Sc_{n+1,r}/ \Sc_{n+1,r}(>\la)$, 
and that 
\begin{align} 
\label{dual right D iso N}
\Hom_R (\D_{n+1}^\sharp(\la), R) \cong \D_{n+1}(\la)^\cas \cong \N_{n+1}(\la) 
\end{align}
as left $\Sc_{n+1,r}$-modules. 

Let 
$\vL_{n+1,r}^+ = \{\la_1, \la_2, \cdots, \la_g\}$ 
be such that 
$i<j$ if $\la_i < \la_j$. 
By a general theory of cellular algebras, 
we obtain a filtration of $(\Sc_{n+1,r}, \Sc_{n+1,r})$-bimodules 
\[ 
\Sc_{n+1,r} = J_1 \supset J_2 \supset \dots  \supset J_g \supset J_{g+1}=0
\]
such that 
$J_i/ J_{i+1} \cong \D_{n+1}(\la_i) \otimes_R \D_{n+1}^\sharp (\la_i)$.
By applying the exact functor $\Res^{n+1}_n$ to this filtration, 
we obtain a filtration of $(\Sc_{n,r}, \Sc_{n+1,r})$-bimodules 
\[ 
\xi \Sc_{n+1,r}
= \xi  J_1 \supset \xi  J_2 \supset \dots \supset \xi J_g  \supset \xi  J_{g+1}=0 
\]
such that 
$\xi J_i /  \xi J_{i+1} \cong \Res^{n+1}_n (\D_{n+1}(\la_i)) \otimes_R \D_{n+1}^\sharp (\la_i)$. 
By (\roi), 
we have that 
$\xi J_i/ \xi J_{i+1} \in \Sc_{n,r} \cmod^\D$ for each $i=1, \dots, g$. 
Thus, by a general theory of quasi-hereditary algebras, 
we have a filtration of left $\Sc_{n+1,r}$-modules 
\begin{align}
\label{filtration Hom xi Scn+1 N}
\Hom_{\Sc_{n,r}} ( \xi \Sc_{n+1,r}, \N_n(\mu) ) 
=N_g \supset N_{g-1} \supset \dots \supset N_1 \supset N_0=0
\end{align}
such that 
$N_i/ N_{i-1} \cong \Hom_{\Sc_{n,r}}(\xi J_i / \xi J_{i+1}, \N_n(\mu))$. 

On the other hand, 
we have the following isomorphisms as $\Sc_{n+1,r}$-modules. 
\begin{align}
&\label{Hom xi Ji / xi Ji+1, N} 
\Hom_{\Sc_{n,r}} (\xi J_i / \xi J_{i+1}, \N_n(\mu) 
\\
\notag
&\cong
\Hom_{\Sc_{n,r}} \big( \Res^{n+1}_n (\D_{n+1}(\la_i))\otimes_R \D^\sharp_{n+1}(\la_i), \N_n(\mu) \big)
\\
\notag
&\cong 
\Hom_R \big( \D_{n+1}^\sharp(\la_i), \Hom_{\Sc_{n,r}} ( \Res_n^{n+1} ( \D_{n+1} ( \la_i ) ), \N_{n}(\mu )) \big)
\\
\notag
&\cong 
\begin{cases} 
\Hom_{R}(\D_{n+1}^\sharp(\la_i), R)& \text{ if   $\la_i = \mu \cup x_i$ for some addable node $x_i$ of $\mu$}
\\
0 & \text{ otherwise },
\end{cases}
\end{align}
where 
the last isomorphism follows from a general theory of quasi-hereditary algebras 
(see e.g. \cite[Proposition A2.2 (\roii)]{D-book}) together with (\roi).  

Suppose  
$\la_i = \mu \cup x_i$ and $\la_j = \mu \cup x_j$ 
for some addable nodes $x_i$, $x_j$ of $\mu$.  
Then, we have that  
$\la_i < \la_j $ 
if and only if 
$x_i \prec x_j$. 
Thus, 
by \eqref{filtration Hom xi Scn+1 N} together with 
\eqref{Hom xi Ji / xi Ji+1, N} and \eqref{dual right D iso N}, 
we obtain (\roiv). 

(\roiii) is obtained by applying the contravariant functor $\cas$ to (\roiv) 
thanks to Lemma \ref{Lemma commute Res cas, Ind cas} (\roii). 
\end{proof}



\section{Some properties of induction and restriction functors}

In this section, 
we study some properties for induction and restriction functors. 
In particular, 
we will prove that 
$\Ind^{n+1}_n$ and $\coInd^{n+1}_n$ are isomorphic. 

From now on, 
throughout of this paper, 
we assume that $R$ is a field, 
and that $Q_k \not=0$ for any $k=1,\dots,r$. 

\para 
\textit{Schur functors.} 
We recall a definition and some properties of the Schur functor 
from $\Sc_{n,r}\cmod$ (resp. $\Sc_{n+1,r} \cmod$) 
to $\He_{n,r} \cmod$ (resp. $\He_{n+1,r} \cmod$).
Put
$\w_n = (\emptyset, \dots, \emptyset, (1, \dots, 1, 0, \dots,0)) \in \vL_{n,r}(\Bm')$ 
and 
$\w_{n+1} = \g(\w_n)  \in \vL_{n+1,r}(\Bm)$.
Then, 
it is clear that 
$M^{\w_n} \cong \He_{n,r}$  as right $\He_{n,r}$-modules 
(resp. $M^{\w_{n+1}} \cong \He_{n+1,r}$ as right $\He_{n+1,r}$-modules). 
Thus, we have that 
$1_{\w_n} \Sc_{n,r} 1_{\w_n} = \End_{\He_{n,r}}(M^{\w_n}) \cong \He_{n,r}$ 
(resp. $1_{\w_{n+1}} \Sc_{n+1,r} 1_{\w_{n+1}} = \End_{\He_{n+1,r}}(M^{\w_{n+1}}) \cong \He_{n+1,r}$) 
as $R$-algebras. 

Through the isomorphism  
$\He_{n,r} \cong 1_{\w_n} \Sc_{n,r} 1_{\w_n}$ 
(resp. $\He_{n+1,r} \cong 1_{\w_{n+1}} \Sc_{n+1,r} 1_{\w_{n+1}}$), 
we can define the exact functor, so called Schur functor,  
$\Om_n : \Sc_{n,r} \cmod \ra \He_{n,r} \cmod$ 
(resp. $\Om_{n+1} : \Sc_{n+1,r} \cmod \ra \He_{n+1,r} \cmod$) 
by 
\begin{align*}
\Om_n=  1_{\w_n} \Sc_{n,r} \otimes_{\Sc_{n,r}} ? 
\quad 
(\text{resp. }\Om_{n+1}
	= 1_{\w_{n+1}} \Sc_{n+1,r} \otimes_{\Sc_{n+1,r}} ?). 
\end{align*}
It is well known that 
$\Om_n$ (resp. $\Om_{n+1}$) is isomorphic to the functor 
$\Hom_{\Sc_{n,r}}(\Sc_{n,r} 1_{\w_n},?)$ 
(resp. $\Hom_{\Sc_{n+1,r}}(\Sc_{n+1,r} 1_{\w_{n+1}},?)$).

We also define the functor 
$\Phi_n : \He_{n,r} \cmod \ra \Sc_{n,r} \cmod$ 
(resp. $\Phi_{n+1} : \He_{n+1,r} \cmod \ra \Sc_{n+1,r} \cmod$) 
by 
\begin{align*}
\Phi_n = \Hom_{\He_{n,r}}(1_{\w_n} \Sc_{n,r}, ?) 
\quad 
(\text{resp. } 
\Phi_{n+1} 
	= \Hom_{\He_{n+1,r}}(1_{\w_{n+1}} \Sc_{n+1,r}, ?).  
\end{align*}
Then,
$\Phi_n$ (resp. $\Phi_{n+1}$) 
is the right adjoint functor of $\Om_n$ (resp. $\Om_{n+1}$), 
and that 
\begin{align}
\label{Om Phi Id}
\Om_n \circ \Phi_n \cong \Id_{n}^\He 
\quad 
(\text{resp. } \Om_{n+1} \circ \Phi_{n+1} \cong \Id_{n+1}^\He),
\end{align} 
where $\Id_n^{\He}$ (resp. $\Id_{n+1}^{\He}$) 
is the identity functor on $\He_{n,r} \cmod$ (resp. on $\He_{n+1,r} \cmod$) 
(see e.g. \cite[I. Theorem 6.8]{ASS}).

Let 
$\Sc_{n,r} \cproj$ (resp. $\Sc_{n+1,r} \cproj$) 
be the full subcategory of $\Sc_{n,r} \cmod$ (resp. $\Sc_{n+1,r} \cmod$) 
consisting of projective objects, 
and 
$I_n : \Sc_{n,r} \cproj \ra \Sc_{n,r} \cmod$ 
(resp. $I_{n+1} : \Sc_{n+1,r} \cproj \ra \Sc_{n+1,r} \cmod$) 
be the canonical embedding functor. 
Thanks to  the double centralizer property between 
$\Sc_{n,r}$ (resp. $\Sc_{n+1,r}$) and $\He_{n,r}$ (resp. $\Sc_{n+1,r}$) 
(see e.g. \cite[Theorem 5.3]{Mat04}), 
an adjunction between $\Om_n$ and $\Phi_n$ (resp. $\Om_{n+1}$ and $\Phi_{n+1}$)   
implies 
the  isomorphism of functors 
\begin{align}
\label{Phi Om I iso Id}
\Phi_n \circ \Om_n \circ I_n \cong \Id_n \circ I_n 
\quad  
(\text{resp. } \Phi_{n+1} \circ \Om_{n+1} \circ I_{n+1} \cong \Id_{n+1} \circ I_{n+1}) 
\end{align}
by \cite[Proposition 4.33]{R}.

Note that 
the isomorphism 
$1_{\w_n} \Sc_{n,r} 1_{\w_n} \cong \He_{n,r}$ 
(resp. $1_{\w_{n+1}} \Sc_{n+1,r} 1_{\w_{n+1}} \cong\nobreak \He_{n+1,r}$) 
is given by $\vf \mapsto \vf(m_{\w_n})$ (resp. $\vf \mapsto \vf(m_{\w_{n+1}}$)), 
we have the following lemma. 

\begin{lem}[{\cite[Proposition 6.3]{W2}}]\
\label{Lemma elements T}
\begin{enumerate}
\item Under the isomorphism $\He_{n,r} \cong 1_{\w_n} \Sc_{n,r} 1_{\w_n}$, we have 
\begin{align*}
&T_0 = 1_{\w_n} F_{(m_{r-1}, r-1)} E_{(m_{r-1}, r-1)} 1_{\w_n} + Q_r 1_{\w_n}, 
\\
&T_i= 1_{\w_n} F_{(i,r)} E_{(i,r)} 1_{\w_n} - q^{-1} 1_{\w_n} \quad (1 \leq  i \leq n-1).
\end{align*}

\item Under the isomorphism $\He_{n+1,r} \cong 1_{\w_{n+1}} \Sc_{n+1,r} 1_{\w_{n+1}}$, we have 
\begin{align*}
&T_0 = 1_{\w_{n+1}} F_{(m_{r-1}, r-1)} E_{(m_{r-1}, r-1)} 1_{\w_{n+1}} + Q_r 1_{\w_{n+1}}, 
\\
&T_i= 1_{\w_{n+1}} F_{(i,r)} E_{(i,r)} 1_{\w_{n+1}} - q^{-1} 1_{\w_{n+1}} \quad (1 \leq  i \leq n-1), 
\\
& T_{n} = 1_{\w_{n+1}} F_{(m_{r-1},r)} \dots F_{(n+1,r)} F_{(n,r)} 
		E_{(n,r)} E_{(n+1,r)} \dots E_{(m_r -1,r)} 1_{\w_{n+1}} 
	- q^{-1} 1_{\w_n}.
\end{align*}

\end{enumerate}
\end{lem}

By Proposition \ref{Prop embedding Scn to Scn+1} and Lemma \ref{Lemma elements T}, 
we have the following corollary. 

\begin{cor}
\label{Cor identification of embedding}
The restriction of $\iota : \Sc_{n,r} \ra \Sc_{n+1,r}$ 
to 
$\He_{n,r} \cong 1_{\w_n} \Sc_{n,r} 1_{\w_n}$ 
coincides with  $\iota^\He : \He_{n,r} \ra \He_{n+1,r}$.
\end{cor}

\para 
Let 
$\HRes^{n+1}_n : \He_{n+1,r} \cmod \ra \He_{n,r} \cmod$ 
be the restriction functor through the injective homomorphism $\iota^{\He}$. 
We define the induction functor  
$\HInd^{n+1}_n : \He_{n,r} \cmod \ra \He_{n+1,r} \cmod$ 
by 
$\HInd^{n+1}_n = \He_{n+1,r} \otimes_{\He_{n,r}} ?$. 
Since 
$\He_{n,r}$ (resp. $\He_{n+1,r}$) 
is a symmetric algebra by \cite{MM}, 
$\HInd^{n+1}_n$ is isomorphic to the functor 
$\Hom_{\He_{n,r}}(\He_{n+1,r}, ?)$ 
(see \cite[Lemma 2.6]{S}). 
Thus, 
we see that 
$\HRes^{n+1}_n$ and $\HInd^{n+1}_n$ are exact, 
and 
$\HInd^{n+1}_n$ is left and right adjoint to $\HRes^{n+1}_n$. 

Recall the anti-involution $\ast$ on $\He_{n,r}$ (resp. $\He_{n+1,r}$), 
and 
we  consider the contravariant functor 
$\cas : \He_{n,r} \cmod \ra \He_{n,r} \cmod$ 
(resp.  $\cas : \He_{n+1,r} \cmod \ra \He_{n+1,r} \cmod$) 
with respect to $\ast$. 
Then, we have the following lemma.

\begin{lem}
We have the following isomorphisms of functors. 
\begin{enumerate}
\label{Lemma commute cas Om HRes HInd}
\item 
$\cas \circ \Om_n \cong \Om_n \circ \cas$ 
$($resp. $\cas \circ \Om_{n+1} \cong \Om_{n+1} \circ \cas)$. 

\item 
$\cas \circ \HRes^{n+1}_n \cong \HRes^{n+1}_n \circ \cas$. 

\item 
$\cas \circ \HInd^{n+1}_n \cong \HInd^{n+1}_n \circ \cas$. 
\end{enumerate}
\end{lem}

\begin{proof}
We can prove that 
\begin{align}
\label{iso Hom dual S H}
\Hom_R(1_{\w_n} \Sc_{n,r}, R) \cong (\Sc_{n,r} 1_{\w_n})^\cas 
\quad 
\text{ as $(\Sc_{n,r}, \He_{n,r})$-bimodules} 
\end{align}
in a similar way as in \eqref{iso Hom dual 1}. 

For $M \in \He_{n,r} \cmod$, 
we have  the following natural isomorphisms 
\begin{align*}
\cas \circ \Om_n (M) 
&= 
\Hom_{R} (1_{\w_n} \Sc_{n,r} \otimes_{\Sc_{n,r}} M, R) 
\\
&\cong 
\Hom_{\Sc_{n,r}}(M, \Hom_R(1_{\w_n} \Sc_{n,r}, R))
\\
&\cong 
\Hom_{\Sc_{n,r}} (M, (\Sc_{n,r} 1_{\w_n})^\cas)  \quad (\text{because of } \eqref{iso Hom dual S H})
\\
&\cong 
\Hom_{\Sc_{n,r}}(\Sc_{n,r} 1_{\w_n}, M^\cas) 
\\
&\cong 
\Om_n \circ \cas (M), 
\end{align*}
and this implies (\roi) 
(it is similar for $\Om_{n+1}$).

The anti-involution $\ast$ on $\He_{n,r}$ is 
the restriction of the anti-involution $\ast$ on $\He_{n+1,r}$. 
Thus, we obtain (\roii). 

For $M \in \He_{n,r} \cmod$, 
we have  the following natural isomorphisms 
\begin{align*}
\cas \circ \HInd^{n+1}_n (M) 
&= 
\Hom_{R}(\He_{n+1,r} \otimes_{\He_{n,r}} M , R) 
\\
&\cong 
\Hom_{\He_{n,r}}(M, \Hom_{R} (\He_{n+1,r}. R))
\\
&\cong 
\Hom_{\He_{n,r}} (M, \He_{n+1}^\cas) 
\\
&\cong 
\Hom_{\He_{n,r}}(\He_{n+1,r}, M^\cas) 
\\
&\cong 
\HInd^{n+1}_n \circ \cas(M), 
\end{align*}
and this implies (\roiii). 
\end{proof}

We  have the following commutative relations 
for restriction, induction and  Schur functors.

\begin{prop}
\label{Prop commute Om Res Phi Ind}
We have the following isomorphisms of functors. 
\begin{enumerate}
\item 
$\Om_n \circ \Res^{n+1}_n \cong \HRes^{n+1}_n \circ \Om_{n+1}  $. 

\item 
$\coInd^{n+1}_n \circ \Phi_n \cong \Phi_{n+1} \circ \HInd^{n+1}_n$. 
\end{enumerate}
\end{prop}

\begin{proof}
By Proposition \ref{Prop embedding Scn to Scn+1} and Corollary \ref{Cor identification of embedding}, 
we see that 
\begin{align}
\label{iso S S 1w S 1w}
\Sc_{n+1,r} \xi \otimes_{\Sc_{n,r}} \Sc_{n,r} 1_{\w_n} 
\cong 
\Sc_{n+1,r} 1_{\w_{n+1}} 
\quad 
\text{ as $(\Sc_{n+1,r}, \He_{n,r})$-bimodules}.
\end{align}

For $M \in \Sc_{n+1,r}\cmod$, 
we have the following natural isomorphisms 
\begin{align*}
\Om_n \circ \Res^{n+1}_n (M) 
&= 
\Hom_{\Sc_{n,r}} (\Sc_{n,r} 1_{\w_n}, \Res^{n+1}_n (M)) 
\\
&\cong 
\Hom_{\Sc_{n+1,r}} ( \Ind^{n+1}_n (\Sc_{n,r} 1_{\w_n}), M)
\\
&= 
\Hom_{\Sc_{n+1,r}} (\Sc_{n+1,r} \xi \otimes_{\Sc_{n,r}} \Sc_{n,r} 1_{\w_n} , M) 
\\
&\cong 
\HRes^{n+1}_n \big(  \Hom_{\Sc_{n+1,r}} (\Sc_{n+1,r} 1_{\w_{n+1}}, M)  \big) 
\quad (\text{because of } \eqref{iso S S 1w S 1w})
\\
&= 
\HRes^{n+1}_n \circ \Om_{n+1}(M), 
\end{align*}
and we obtain (\roi). 

We easily see that 
$\coInd^{n+1}_n \circ \Phi_n$
(resp. $\Phi_{n+1} \circ \HInd^{n+1}_n$) 
is right adjoint to 
$\Om_n \circ \Res^{n+1}_n$
(resp.  $\HRes^{n+1}_n \circ \Om_{n+1}$).  
Then, 
by the uniqueness of the adjoint functor together with (\roi),  
we obtain (\roii). 
\end{proof}

\para 
The rest of this section, 
we will prove the isomorphism of functors $\Ind^{n+1}_n \cong \coInd^{n+1}_n$. 
Our strategy is using the good properties 
for standard and costandard modules of quasi-hereditary algebras. 
For  $\Sc_{n,r} \cmod^\D$, 
we have the following two lemmas.

\begin{lem}
\label{Lemma exact Ind D}
For $M,N,L \in \Sc_{n,r} \cmod^\D$, 
and an exact sequence  
\[
 0 \ra N \ra M \ra L \ra 0 
\] 
as $\Sc_{n,r}$-modules, 
we have the exact sequence 
\[ 
0 \ra \Ind^{n+1}_n (N) \ra \Ind^{n+1}_n (M) \ra \Ind^{n+1}_n (L) \ra 0. 
\] 
\end{lem}

\begin{proof}
By applying the functor $\cas$ to this sequence, 
we have the exact sequence 
$0 \ra L^\cas \ra M^\cas \ra N^\cas \ra 0$. 
Since $\xi \Sc_{n+1,r} \in \Sc_{n,r}^\D$ and $L^\cas \in \Sc_{n,r} \cmod^\N$, 
we have 
$\Ext^1_{\Sc_{n,r}}(\xi \Sc_{n+1,r}, L^\cas)=0$ 
by \cite[Proposition A2.2]{D-book}. 
Thus, we have the exact sequence 
$0 \ra \coInd^{n+1}_n (L^\cas) \ra \coInd^{n+1}_n (M^\cas) \ra \coInd^{n+1}_n (N^\cas) \ra 0$. 
By applying the functor $\cas$ to this sequence 
together with Lemma \ref{Lemma commute Res cas, Ind cas} (\roii), 
we have the exact sequence 
$ 0 \ra \Ind^{n+1}_n (N) \ra \Ind^{n+1}_n (M) \ra \Ind^{n+1}_n (L) \ra 0$. 
\end{proof}

\begin{lem}
\label{Lemma comparsion dimensions}
For $M \in \Sc_{n,r} \cmod^\D$, we have the following. 
\begin{enumerate}
\item 
$\dim \Ind^{n+1}_n (M) \geq \dim \coInd^{n+1}_n (M)$. 

\item 
$\dim \HInd^{n+1}_n \circ \Om_n(M) \geq \dim \Om_{n+1} \circ \coInd^{n+1}_n(M)$.   
\end{enumerate}
\end{lem}

\begin{proof}
By Lemma \ref{Lemma Ind D Ind N in K0}, 
we have that 
$\dim \Ind^{n+1}_n (\D_n(\la)) = \dim \Ind^{n+1}_n (\N_n(\la))$ 
for $\la \in \vL_{n,r}^+$. 
Since 
$\Ind^{n+1}_n (\N_n(\la)) \cong \Ind^{n+1}_n (\D_n(\la)^\cas) \cong \cas \circ \coInd^{n+1}_n (\D_n(\la))$ 
by Lemma \ref{Lemma commute Res cas, Ind cas} (\roii), 
we have that 
$\dim \Ind^{n+1}_n (\D_n(\la)) = \dim \coInd^{n+1}_n (\D_n(\la))$. 

Since $M \in \Sc_{n,r} \cmod^\D$, 
we can take a exact sequence 
\begin{align}
\label{seq N1 M N2} 
0 \ra N_1 \ra M \ra N_2  \ra 0  
\text{ such that } 
N_1, N_2 \in \Sc_{n,r}\cmod^\D.
\end{align} 

By Lemma \ref{Lemma exact Ind D}, 
we have the exact sequence 
\[
0 \ra \Ind^{n+1}_n (N_1) \ra \Ind^{n+1}_n (M) \ra \Ind^{n+1}_n (N_2) \ra 0,
\] 
and we have 
\begin{align}
\label{dim Ind M N1 N2}
\dim \Ind^{n+1}_n (M) = \dim \Ind^{n+1}_n (N_1) + \dim \Ind^{n+1}_n (N_2).
\end{align} 

On the other hand, 
by applying the left exact functor $\coInd^{n+1}_n$ to the sequence \eqref{seq N1 M N2}, 
we have the exact sequence 
$0 \ra \coInd^{n+1}_n (N_1) \ra \coInd^{n+1}_n (M) \ra \coInd^{n+1}_n (N_2)$,  
and we have 
\begin{align} 
\label{dim coInd M N1 N2}
\dim \coInd^{n+1}_n (M) \leq \dim \coInd^{n+1}_n (N_1) + \dim \coInd^{n+1}_n (N_2).
\end{align} 
By the induction on the length of $\D$-filtration of $M$, 
\eqref{dim Ind M N1 N2} and \eqref{dim coInd M N1 N2} 
imply (\roi).

By a similar way as in (\roi), 
we can prove that 
\begin{align}
\label{dim Om Ind M Om coInd M} 
\dim \Om_{n+1} \circ \Ind^{n+1}_n (M) \geq \dim \Om_{n+1} \circ \coInd^{n+1}_n (M). 
\end{align}
On the other hand, 
it is known that 
$\Om_n(\D_n(\la))$ is isomorphic to the Specht module $S^\la$ defined in \cite{DJM98}.  
Thus, 
by Theorem \ref{Thm Res Ind Weyl} (\roiii) and \cite[Cororally 1.10]{AM}, 
we have 
\[ 
\dim \HInd^{n+1}_n \circ \Om_n(M) = \dim \Om_{n+1} \circ \Ind^{n+1}_n(M).
\]
Combining this equation with \eqref{dim Om Ind M Om coInd M}, 
we obtain (\roii). 
\end{proof}

We prepare the following general results. 

\begin{lem}
\label{Lemma lift morphism of functors}
Let $\A$, $\B$ be finite dimensional algebras over a field, 
$I : \A \cproj \ra \A \cmod$ 
be the canonical embedding functor, 
and 
$F,G$ be functors from $\A \cmod $ to $\B \cmod$. 
Then we have the following. 

\begin{enumerate}
\item 
If $F$ is a right exact functor, 
the homomorphism of vector spaces 
\[ 
\Hom (F,G) \ra \Hom (F \circ I, G \circ I) 
\text{ given by } \nu \mapsto \nu 1_{I}  
\]
is an isomorphism. 

\item 
Assume that $F$ and $G$ are right exact functors. 
If $F \circ I \cong G \circ I$, 
we have $F \cong G$.  

\item 
Assume that $\A$ is a quasi-hereditary algebra, 
and let $I^\D : \A \cmod^\D \ra \A \cmod$ be the canonical embedding functor. 
If $F$ is a right exact functor, 
the homomorphism of vector spaces 
\[ 
\Hom (F,G) \ra \Hom (F \circ I^\D, G \circ I^\D) 
\text{ given by } \nu \mapsto \nu 1_{I^\D}  
\]
is an isomorphism. 

\end{enumerate}
\end{lem}

\begin{proof}
We can prove (\roi) and (\roii) in a similar way as in \cite[Lemma 1.2]{S}. 
We prove (\roiii). 

Since any projective $\A$-module is an object of $\A \cmod^\D$, 
we have $I^\D \circ I \cong I$, and 
we have the following commutative diagram. 
\begin{align*}
\xymatrix{
\Hom(F,G) \ar[rr]^{\nu \mapsto \nu 1_I} \ar[dr]_{\nu \mapsto \nu 1_{I^\D}} 
	&& \Hom(F \circ I, G \circ I) 
\\
& \Hom(F \circ I^\D, G \circ I^\D) \ar[ur]_{\t \mapsto \t 1_I}
}
\end{align*}
By (\roi), 
$\Hom(F,G) \ra \Hom (F \circ I, G \circ I)$ is an isomorphism. 
We can also prove that 
$\Hom(F\circ I^\D, G \circ I^\D) \ra \Hom (F \circ I, G \circ I)$ 
is an isomorphism in a similar way. 
Thus, the above diagram implies (\roiii). 
\end{proof}

\para 
Let 
$I_n^\D : \Sc_{n,r} \cmod^\D \ra \Sc_{n,r} \cmod $ 
(resp. $I_n^\N : \Sc_{n,r} \cmod^\N \ra \Sc_{n,r} \cmod $) 
be the canonical embedding functor. 
Then, we have the following proposition, 
which is a key step to prove the isomorphism $\Ind^{n+1}_n \cong \coInd^{n+1}_n$.

\begin{prop}
We have an isomorphism of functors 
\[
 \Ind^{n+1}_n \circ I_n^\D \cong \coInd^{n+1}_n \circ I_n^\D. 
\]
\end{prop}

\begin{proof}
By Proposition \ref{Prop commute Om Res Phi Ind} (\roii) and \eqref{Om Phi Id}, 
we have 
\begin{align*}
\Om_{n+1} \circ \coInd^{n+1}_n \circ \Phi_n 
&\cong 
\Om_{n+1} \circ \Phi_{n+1} \circ \HInd^{n+1}_n
\\
&\cong 
\HInd^{n+1}_n.
\end{align*} 
These isomorphisms together with \eqref{Phi Om I iso Id} imply that  
\begin{align}
\label{HInd Om I iso Om coInd I}
\begin{split}
\HInd^{n+1}_n \circ \Om_n \circ I_n 
&\cong 
\Om_{n+1} \circ \coInd^{n+1}_n \circ \Phi_n  \circ \Om_n \circ I_n 
\\
&\cong 
\Om_{n+1} \circ \coInd^{n+1}_n \circ I_n.
\end{split}
\end{align}
Thus, there exists a functorial isomorphism 
\[
\wt{\nu} : \HInd^{n+1}_n \circ \Om_n \circ I_n  \ra \Om_{n+1} \circ \coInd^{n+1}_n \circ I_n.
\] 
Since $\HInd^{n+1}_n \circ \Om_n$ is an exact functor, 
there exists the unique morphism 
\[
\nu : \HInd^{n+1}_n \circ \Om_n \ra \Om_{n+1} \circ \coInd^{n+1}_n  
\text{ such that } 
\nu 1_{I_n} =\wt{\nu} 
\] 
by Lemma \ref{Lemma lift morphism of functors} (\roi). 

We prove that 
$\nu 1_{I_n^\D} : \HInd^{n+1}_n \circ \Om_n \circ I_n^\D \ra \Om_{n+1} \circ \coInd^{n+1}_n \circ I_n^\D$
gives an isomorphism of functors. 
Note that the global dimension of $\Sc_{n,r}$ is finite since $\Sc_{n,r}$ is a quasi-hereditary algebra. 
Thus, 
for any $M \in \Sc_{n,r} \cmod^\D$, 
we can take a projective resolution 
\begin{align*}
0 \ra P_k \xra{d_k} \dots \xra{d_2} P_1 \xra{d_1} P_0 \xra{d_0} M \ra 0 
\end{align*} 
such that 
$k$ is equal to the projective dimension of $M$ (denoted by $\pdim M$). 
By an induction on $\pdim M$, 
we prove that $\nu 1_{I^\D_n}(M)$ is an isomorphism. 

When $\pdim M =0$, 
we have $\nu(M) = \wt{\nu}(M)$ since $M$ is projective. 
Thus, $\nu(M)$ is an isomorphism. 

Assume that $\pdim M >0$. 
For the short exact sequence 
\begin{align}
\label{short exact seq Ker P0 M}
0 \ra \Ker d_0 \ra P_0 \xra{d_0} M \ra 0,
\end{align} 
we have that 
$\Ker d_0 \in \Sc_{n,r} \cmod^\D$ 
by \cite[Proposition A2.2 (\rov)]{D-book}. 
Moreover, we have $\pdim \Ker d_0 \leq \pdim M-1$. 
By applying the functors 
$F:= \HInd^{n+1}_n \circ \Om_n$ and $G:= \Om_{n+1} \circ \coInd^{n+1}_n$ 
to \eqref{short exact seq Ker P0 M}, 
we have the following commutative diagram 
\begin{align*}
\xymatrix{
0 \ar[r]  
& F (\Ker d_0) \ar[r] \ar[d]_{\nu (\Ker d_0)} 
& F (P_0) \ar[r] \ar[d]_{\nu(P_0)} 
& F (M) \ar[r] \ar[d]_{\nu(M)} 
& 0
\\
0 \ar[r] 
& G (\Ker d_0) \ar[r]
& G  (P_0) \ar[r] 
& G  (M)
}
\end{align*}
such that 
each row is exact. 
Note that 
$\nu (\Ker d_0)$ (resp. $\nu(P_0)$) 
is an isomorphism by the assumption of induction (resp. the fact $P_0$ is projective). 
Then, 
the above diagram implies that $\nu(M)$ is injective. 
Thus, $\nu(M)$ is an isomorphism by Lemma \ref{Lemma comparsion dimensions} (\roii). 
Now we proved that 
$\nu 1_{I^\D_n}$ gives the isomorphism 
\begin{align}
\label{Hind Om ID iso Om coInd ID}
\HInd^{n+1}_n \circ \Om_n \circ I_n^\D 
\cong 
 \Om_{n+1} \circ \coInd^{n+1}_n \circ I_n^\D.
\end{align}

Next, we prove that 
$\nu 1_{I_n^\N} : \HInd^{n+1}_n \circ \Om_n \circ I_n^\N \ra \Om_{n+1} \circ \coInd^{n+1}_n \circ I_n^\N$
gives an isomorphism of functors. 
By \cite[Proposition 4.4]{D-book}, 
for $N \in \Sc_{n,r} \cmod^\N$, 
we can take the following exact sequence 
\begin{align*}
0 \ra T_k \xra{d'_k} \dots \xra{d'_2}  T_1 \xra{d'_1} T_0 \xra{d'_0} N \ra 0, 
\end{align*}
such that $T_i$ is a (characteristic) tilting module, and that $\Ker d'_i \in \Sc_{n,r} \cmod^\N$ for each $i=0,\dots,k$.  
By applying the functors 
$F:= \HInd^{n+1}_n \circ \Om_n$ and $G:= \Om_{n+1} \circ \coInd^{n+1}_n$ 
to this exact sequence, 
we have the following commutative diagram 
\begin{align*}
\xymatrix{
& F (T_1) \ar[r] \ar[d]_{\nu (T_1)} 
& F (T_0) \ar[r] \ar[d]_{\nu(T_0)} 
& F (N) \ar[r] \ar[d]_{\nu(N)} 
& 0
\\
& G (T_1) \ar[r]
& G  (T_0) \ar[r] 
& G  (N) \ar[r]
& 0
}
\end{align*}
such that each row is exact. 
(The exactness of the second row comes from the fact 
$\xi \Sc_{n+1,r} \in \Sc_{n,r} \cmod^\D$ and  $\Ker d'_i \in \Sc_{n,r} \cmod^\N$ 
thanks to \cite[Proposition A2.2 (\roii)]{D-book}.) 
We already proved that $\nu(T_1)$ and $\nu(T_0)$ are isomorphisms 
since $T_0, T_1 \in \Sc_{n,r} \cmod^\D$.  
Thus, the above diagram implies that 
$\nu(N)$ is an isomorphism. 
Now, we proved that 
$\nu 1_{I^\N_n}$ gives the isomorphism 
\begin{align}
\label{HInd Om IN iso Om coInd IN}
\HInd^{n+1}_n \circ \Om_n \circ I_n^\N
\cong 
 \Om_{n+1} \circ \coInd^{n+1}_n \circ I_n^\N.
\end{align}
Note that 
$\cas \circ I_n^\N \circ \cas \circ I_n^\D \cong I_n^\D$, 
we have 
\begin{align*}
&\HInd^{n+1}_n \circ \Om_n \circ I_n^\D 
\\
&\cong
\HInd^{n+1}_n \circ \Om_n \circ  \cas \circ I_n^\N \circ \cas \circ I_n^\D 
\\
&\cong
\cas \circ \HInd^{n+1}_n \circ \Om_n \circ I_n^\N \circ \cas \circ I_n^\D 
\quad (\because \text{Lemma } \ref{Lemma commute cas Om HRes HInd})
\\
&\cong 
\cas \circ \Om_{n+1} \circ \coInd^{n+1}_n \circ I_n^\N \circ \cas \circ I_n^\D 
\quad (\because \eqref{HInd Om IN iso Om coInd IN})
\\
&\cong
\Om_{n+1} \circ \Ind^{n+1}_n \circ \cas \circ I_n^\N \circ \cas \circ I_n^\D 
\quad 
(\because \text{Lemma } \ref{Lemma commute cas Om HRes HInd} 
	\text{ and Lemma } \ref{Lemma commute Res cas, Ind cas})
\\
&\cong 
\Om_{n+1} \circ \Ind^{n+1}_n \circ I_n^\D. 
\end{align*}
Combining these isomorphisms with Proposition \ref{Prop commute Om Res Phi Ind} (\roii), 
we have 
\begin{align}
\label{coInd Phi Om ID iso Phi Om Ind ID}
\coInd^{n+1}_n \circ \Phi_n \circ \Om_n \circ I_n^\D 
&\cong 
\Phi_{n+1} \circ \HInd^{n+1}_n \circ \Om_n \circ I_n^\D
\\
&\cong 
\Phi_{n+1} \circ \Om_{n+1} \circ \Ind^{n+1}_n \circ I_n^\D. 
\notag 
\end{align}
Note that 
$I_n^\D \circ I_n \cong I_n$, and that the functor $\Ind^{n+1}_n$ preserves projectivity  
since $\Ind^{n+1}_n$ has an exact right adjoint functor $\Res^{n+1}_n$. 
Then, 
\eqref{coInd Phi Om ID iso Phi Om Ind ID} together with \eqref{Phi Om I iso Id} implies 
\begin{align*}
\Ind^{n+1}_n \circ I_n \cong \coInd^{n+1}_n \circ I_n. 
\end{align*}
From this isomorphism, 
we can prove that 
\begin{align*}
\Ind^{n+1}_n \circ I_n^\D \cong \coInd^{n+1}_n \circ I_n^\D 
\end{align*}
by using Lemma \ref{Lemma exact Ind D} and Lemma \ref{Lemma comparsion dimensions} (\roi) 
in a similar way as in the proof of \eqref{Hind Om ID iso Om coInd ID}.
\end{proof}

\para 
By Theorem \ref{Thm Res Ind Weyl} (\roiii) and Lemma \ref{Lemma exact Ind D}, 
we see that 
$\Ind^{n+1}_n (M) \in \Sc_{n+1,r} \cmod^\D$ 
for $M \in \Sc_{n,r} \cmod^\D$. 
Thus, 
$\Ind^{n+1}_n \circ I_n^\D \cong \coInd^{n+1}_n \circ I_n^\D$ 
gives a functor from $\Sc_{n,r} \cmod^\D$ to $\Sc_{n+1,r} \cmod^\D$. 
On the other hand, 
by Theorem \ref{Thm Res Ind Weyl} (\roi), 
$\Res^{n+1}_n \circ I_{n+1}^\D$  gives a functor from $\Sc_{n+1,r} \cmod^\D$ to $\Sc_{n,r} \cmod^\D$.  
Moreover, 
we see that 
$\Ind^{n+1}_n \circ I_n^\D$ 
is left and right adjoint to 
$\Res^{n+1}_n \circ I_{n+1}^\D$. 
This adjunction induces the following theorem.

\begin{thm}
\label{Thm iso Ind coInd}
We have the following isomorphism of functors 
\[ 
\Ind^{n+1}_n \cong \coInd^{n+1}_n. 
\]
\end{thm}

\begin{proof}
By the uniqueness of the adjoint functor, 
it is enough to show that 
$\Ind^{n+1}_n$ is right adjoint to $\Res^{n+1}_n$ 
since 
$\coInd^{n+1}_n$ is right adjoint to $\Res^{n+1}_n$. 

Put 
$F= \Ind^{n+1}_n$ and $E= \Res^{n+1}_n$. 
Since 
$F \circ I_n^\D$ is right adjoint to $E \circ I_{n+1}^\D$, 
there exist the morphisms of functors 
$\wt{\ve} : E \circ I_{n+1}^\D \circ F \circ I_n^\D \ra \Id_n \circ I_n^\D$ (unit) 
and 
$\wt{\eta} : \Id_{n+1} \circ I_{n+1}^\D \ra F \circ I_n^\D \circ E \circ I_{n+1}^\D$ 
(counit) such that 
$(\wt{\ve} \, 1_{E \circ I_{n+1}^\D}) \circ (1_{E\circ I_{n+1}^\D} \wt{\eta}) = 1_{E \circ I_{n+1}^\D}$,
and that 
$(1_{F \circ I_n^\D} \wt{\ve}) \circ (\wt{\eta} \, 1_{F \circ I_n^\D}) = 1_{F \circ I_n^\D}$, 
where we regard the functor $\Id_n \circ I_n^\D$ (resp. $\Id_{n+1} \circ I_{n+1}^\D$) 
as the identity functor on $\Sc_{n,r} \cmod^\D$ (resp. $\Sc_{n+1,r} \cmod^\D$). 
Note that $I_{n+1}^\D \circ F \circ I_n^\D \cong F \circ I_n^\D$ etc., 
we can write simply as 
$\wt{\ve} : E \circ F \circ I_n^\D \ra \Id_n \circ I_n^\D$ 
and 
$\wt{\eta} : \Id_{n+1} \circ I_{n+1}^\D \ra F \circ E \circ I_{n+1}^\D$ 
such that 
\begin{align}
&(\wt{\ve}\, 1_E 1_{I_{n+1}^\D}) \circ (1_E \wt{\eta}) = 1_E 1_{I_{n+1}^\D},
\\
&(1_F \wt{\ve}) \circ (\wt{\eta}\, 1_F 1_{I_n^\D}) = 1_F 1_{I_n^\D}. 
\end{align}
By Lemma \ref{Lemma lift morphism of functors} (\roiii), 
there exist the morphisms of functors  
\[ 
\ve : E \circ F \ra \Id_n, 
\quad 
\eta : \Id_{n+1} \ra F \circ E
\]
such that 
$\ve 1_{I_n^\D} = \wt{\ve}$ and $\eta 1_{I_{n+1}^\D} = \wt{\eta}$. 
Moreover, we have 
\begin{align}
\label{adjoint 1}
\begin{split}
((\ve 1_E) \circ (1_E \, \eta)) 1_{I_{n+1}^\D} 
&= 
((\ve 1_E) \circ (1_E \, \eta))) ( 1_{I_{n+1}^\D} \circ 1_{I_{n+1}^\D} ) 
\\
&= 
(\ve 1_E 1_{I_{n+1}^\D}) \circ (1_E  \, \eta 1_{I_{n+1}^\D})
\\
&= (\wt{\ve} 1_E 1_{I_{n+1}^\D}) \circ (1_E \wt{\eta})
\\
&= 1_E 1_{I_{n+1}^\D}. 
\end{split}
\end{align} 
Similarly, we have 
\begin{align}
\label{adjoint 2}
\begin{split}
((1_F \ve) \circ (\eta 1_F)) 1_{I_n^\D} 
&= (1_F \ve 1_{I_n^\D}) \circ (\eta 1_F 1_{I_n^\D}) 
\\
&= (1_F \, \wt{\ve}) \circ (\wt{\eta} 1_F 1_{I_n^\D}) 
\\
&= 1_F 1_{I_n^\D}. 
\end{split}
\end{align}
By \eqref{adjoint 1} and \eqref{adjoint 2} 
together with Lemma \ref{Lemma lift morphism of functors} (\roiii), 
we have 
\[ 
(\ve 1_E) \circ (1_E \, \eta) = 1_E, 
\quad 
(1_F \ve) \circ (\eta 1_F) = 1_F, 
\]
and 
$F$ is right adjoint to $E$. 
\end{proof}

Now we have the following properties of induction and restriction functors. 

\begin{cor}
\label{Cor properties Res Ind}
We have the following. 
\begin{enumerate}
\item 
$\Res^{n+1}_n$ and $\Ind^{n+1}_n$ are exact. 

\item 
$\Ind^{n+1}_n$ is left and right adjoint to $\Res^{n+1}_n$. 

\item 
There exist  isomorphisms of functors 
\[
 \Res^{n+1}_n \circ \cas \cong \cas \circ \Res^{n+1}_n, 
 \quad 
 \Ind^{n+1}_n \circ \cas \cong \cas \circ \Ind^{n+1}_n. 
\]

\item 
There exist  isomorphisms of functors 
\[ 
\Om_n \circ \Res^{n+1}_n \cong \HRes^{n+1}_n \circ \Om_{n+1}, 
\quad 
\Om_{n+1} \circ \Ind^{n+1}_n \cong \HInd^{n+1}_ \circ \Om_n.
\]
\end{enumerate}
\end{cor}

\begin{proof}
(\roi) and (\roii) are obtained from the definitions and Theorem \ref{Thm iso Ind coInd}. 
(\roiii) is obtained from Lemma \ref{Lemma commute Res cas, Ind cas} 
and Theorem \ref{Thm iso Ind coInd}. 
The first isomorphism in (\roiv) is Proposition \ref{Prop commute Om Res Phi Ind} (\roi). 
By \eqref{HInd Om I iso Om coInd I} and Theorem \ref{Thm iso Ind coInd}, 
we have 
\[ 
\Om_{n+1} \circ \Ind^{n+1}_n \circ I_n \cong \HInd^{n+1}_n \circ \Om_n \circ I_n. 
\]
Thus, 
by Lemma \ref{Lemma lift morphism of functors} (\roii), 
we have 
$\Om_{n+1} \circ \Ind^{n+1}_n \cong \HInd^{n+1}_ \circ \Om_n$. 
\end{proof}



\section{Refinements of induction and restriction functors}

In this section, 
we refine the induction and restriction functors which are defined in the previous sections. 
As an application, we categorify a Fock space by using categories  $\Sc_{n,r} \cmod$ ($n \geq 0$). 

Throughout this section, 
we assume that $R$ is a field, 
and 
we also assume the following conditions for parameters. 
\begin{itemize}
\item 
There exists the minimum positive integer $e$ such that 
$1 + (q^2) + (q^2)^2 + \dots + (q^2)^{e-1}=0$. 
\item 
There exists an integer $s_i \in \ZZ$ such that 
$Q_i=(q^2)^{s_i}$ for each $i=1,\dots, r$. 
\end{itemize}
Thanks to 
\cite[Theorem 1.5]{DM}, 
these assumptions make no loss of generality in representation theory of cyclotomic $q$-Schur algebras.  

We also remark that 
$\Sc_{n,r} \cmod$ does not depend on 
a choice of $\Bm=(m_1,\cdots,m_r) \in \ZZ_{>0}^r$ such that 
$m_k \geq n$  for any $k=1,\dots, r$ up to Morita equivalence 
(see Remark \ref{Remark Morita equivalent}). 
Then, 
for each $n$, 
we take suitable $\Bm$ and $\Bm'$ 
to consider the induction and restriction functors between $\Sc_{n,r} \cmod$ and $\Sc_{n+1,r} \cmod$ 
as in the previous sections.

%
%

\para 
For $x=(a,b,c) \in \ZZ_{>0} \times \ZZ_{>0} \times \{1,\dots,r\}$, 
we define the residue  of $x$ by 
\[ 
\res(x)= (q^2)^{b-a} Q_c = (q^{2})^{b-a + s_c}. 
\]
For $x \in \ZZ_{>0} \times \ZZ_{>0} \times \{1,\dots,r\}$,  
we say  that 
$x$ is $i$-node if $\res (x)=(q^2)^{i}$, 
where we can regard $i$ as an element of $\ZZ/ e \ZZ$ 
since $(q^2)^{i + k e} = (q^2)^i$ for any $k \in \ZZ$ from the assumption for parameters. 
We also say that 
$x$ is removable (resp. addable) $i$-node of $\la \in \vL_{n,r}^+$, 
if $x$ is $i$-node and removable (resp. addable) node of $\la$.  

For $ \la \in \vL_{n,r}^+$, 
put $r(\la) =(r_0(\la), r_1(\la), \dots, r_{e-1}(\la)) \in \ZZ^e_{\geq 0}$, 
where 
$r_i(\la)$ is the number of $i$-node in $[\la]$. 
Then it is known that 
the classification of blocks of $\Sc_{n,r}$ in \cite{LM} as follows. 

\begin{thm}[{\cite{LM}}] 
\label{Theorem LM} 
For $ \la,\mu \in \vL_{n,r}^+$, 
$\D_n(\la)$ and $\D_n(\mu)$ belong to the same block of $\Sc_{n,r}$ 
if and only if 
$r(\la) = r(\mu)$. 
\end{thm}

\para 
Put 
\[ R_{n,e} = \{ a = (a_0,a_1,\dots,a_{e-1}) \in \ZZ^e \,|\, a= r(\la) \text{ for some } \la \in \vL_{n,r}^+\}.\]
Then we have a bijection 
between 
$R_{n,e}$ and the set of blocks of $\Sc_{n,r}$ by Theorem \ref{Theorem LM}. 
By using this bijection, 
for $a=(a_0, a_1,\dots, a_{e-1}) \in \ZZ^e$ such that $\sum_{j=0}^{e-1} a_j =n$, 
we define the functor $1_a : \Sc_{n,r} \cmod \ra \Sc_{n,r} \cmod$ as 
the projection to the corresponding block if $a \in R_{n,e}$, 
and $0$ if $a \not\in R_{n,e}$. 

We define a refinement of $\Res^{n+1}_n$ and $\Ind^{n+1}_n$ as follows.  
For $i \in \ZZ/ e \ZZ$, 
put 
\begin{align*}
& \iRes^{n+1}_n = \bigoplus_{a \in R_{n+1,e}} 1 _{a-i} \circ \Res^{n+1}_n \circ 1_a,
\\
& \iInd^{n+1}_n = \bigoplus_{a \in R_{n,e}} 1_{a +i} \circ \Ind^{n+1}_n \circ 1_a,
\end{align*}
where 
$a \pm i= (a_0, \dots, a_{i-1}, a_i \pm 1, a_{i+1}, \dots, a_{e-1})$ 
for $a=(a_0,\dots,a_{e-1}) \in \ZZ^e$. 
Then, we have 
$\Res^{n+1}_n = \bigoplus_{i \in \ZZ/ e\ZZ} \iRes^{n+1}_n$ 
and 
$\Ind^{n+1}_n = \bigoplus_{i \in \ZZ / e \ZZ} \iInd^{n+1}_n$.
From the definition together with Theorem \ref{Thm Res Ind Weyl} and Theorem \ref{Thm iso Ind coInd}, 
we have the following corollary.

\begin{cor}\
\label{Cor iRes iInd on Weyl}
\begin{enumerate}
\item 
For $\la \in \vL_{n+1,r}^+$, 
there exists a filtration of $\Sc_{n,r}$-modules 
\[ 
\iRes^{n+1}_n (\D_{n+1}(\la)) = M_1 \supset M_2 \supset \dots \supset M_k \supset M_{k+1}=0, 
\]
such that 
$M_i / M_{i+1} \cong \D_n (\la \setminus x_i)$, 
where 
$x_1,x_2, \dots, x_k$ are all removable \break $i$-nodes of $\la$ such that 
$x_1 \succ x_2 \succ \dots \succ x_k$.

\item 
For $\mu \in \vL_{n,r}^+$, there exists a filtration of $\Sc_{n+1,r}$-modules 
\[ 
\iInd^{n+1}_n (\D_n(\mu)) = M_1 \supset M_2 \supset \dots \supset M_k \supset M_{k+1}=0 
\]
such that $M_i / M_{i+1} \cong \D_{n+1}(\mu \cup x_i)$, 
where 
$x_1, x_2, \dots, x_k$ are all addable \break $i$-nodes of $\mu$ 
such that 
$x_k \succ x_{k-1} \succ \dots \succ x_1$. 
\end{enumerate}
\end{cor}

\para 
Put $\Bs =(s_1,s_2,\dots,s_r)$. 
The Fock space with multi-charge $\Bs$ is the $\CC$-vector space 
\[ 
\CF[\Bs] = \bigoplus_{n \in \ZZ_{\geq0}} \bigoplus_{\la \in \vL_{n,r}^+} \CC | \la, \Bs \ran 
\] 
with distinguished basis 
$\{|\la,\Bs \ran \,|\, \la \in \vL_{n,r}^+,\, n \in \ZZ_{\geq 0}\}$ 
which admits an integrable $\wh{\Fsl}_e$-module structure 
with the Chevalley generators acting as follows (cf. \cite{JMMO}): 
for $i \in \ZZ / e \ZZ$, 
\[ 
e_i \cdot |\la, \Bs \ran = \sum_{\mu= \la \setminus x \atop \res (x) =(q^2)^i} | \mu, \Bs \ran, 
\quad 
f_i \cdot |\la, \Bs \ran = \sum_{\mu= \la \cup x \atop \res (x) =(q^2)^i} | \mu, \Bs \ran.  
\]

\para 
Put 
\[
 \iRes = \bigoplus_{n \in \ZZ_{\geq 0}} \iRes^{n+1}_n, 
\quad 
\iInd = \bigoplus_{n \in \ZZ_{\geq 0}} \iInd^{n+1}_n.
\]
Since $\iRes$ and $\iInd$ are exact functors from 
$\bigoplus_{n \geq 0} \Sc_{n,r} \cmod$ to itself, 
these functors imply the well-defined action 
on $\CC \otimes_{\ZZ} K_0 ( \bigoplus_{n \geq 0}\Sc_{n,r} \cmod)$. 
Thanks to Corollary \ref{Cor iRes iInd on Weyl}, 
for $ \la \in \vL_{n,r}^+$, 
we have that 
\[ 
\iRes \cdot [\D_n(\la)]  = \sum_{\mu= \la \setminus x \atop \res (x) =(q^2)^i} [\D_{n-1}(\mu) ], 
\quad 
\iInd \cdot [\D_n(\la)] = \sum_{\mu= \la \cup x \atop \res (x) =(q^2)^i}  [\D_{n+1}(\mu) ].  
\]
Note that 
$\{[\D_n(\la)] \,|\, \la \in \vL_{n,r}^+, n \in \ZZ_{\geq 0}\}$ 
gives an $\CC$-basis of $\CC \otimes_{\ZZ} K_0 ( \bigoplus_{n \geq 0}\Sc_{n,r} \cmod)$. 
Then, we have the following corollary. 

\begin{cor}
\label{Cor categorification of Fock}
The exact functors $\iRes$ and $\iInd$ ($i \in \ZZ / e\ZZ$) 
give the action of  $\wh{\Fsl}_e$ on $\CC \otimes_{\ZZ} K_0 ( \bigoplus_{n \geq 0}\Sc_{n,r} \cmod)$, 
where 
$\iRes$ (resp. $\iInd$) is corresponding to the action of the Chevalley generators $e_i$ (resp. $f_i$) of $\wh{\Fsl}_e$. 
Moreover, 
by the correspondence 
$[\D_n(\la)] \mapsto |\la,\Bs\ran$ ($\la \in \vL_{n,r}^+, n \in \ZZ_{\geq 0}$) of basis, 
$\CC \otimes_{\ZZ} K_0 ( \bigoplus_{n \geq 0}\Sc_{n,r} \cmod)$ 
is isomorphic to the Fock space $\CF[\Bs]$ 
as $\wh{\Fsl}_e$-modules. 
\end{cor}

\remarks 
(\roi).  
The results in this section do not depend on the characteristic of the ground field $R$, 
namely depend only $e$ and the multi-charge $\Bs=(s_1,\dots,s_r)$. 
\\
(\roii). 
By the lifting arguments from the module categories of Ariki-Koike algebras as in 
\cite[Section 5]{S}, 
we obtain the $\wh{\Fsl}_e$-categorification in the sense of \cite{R2} in our setting. 



\section{Relations with category $\CO$ of rational Cherednik algebras} 
In this section, 
we assume that $R=\CC$. 
We give a relation between our induction and restriction functors for cyclotomic $q$-Schur algebras 
and parabolic induction and restriction functors for rational Cherednik algebras given in \cite{BE}. 

\para 
Let $\CH_{n,r}$ be the rational Cherednik algebra 
associated to $\FS_n \ltimes (\ZZ / r \ZZ)^n $ with the parameters $\Bc$ 
(see \cite{R} for definition and parameters $\Bc$), 
and 
$\CO_{n,r}$ be the category $\CO$ of $\CH_{n,r}$ defined in \cite{GGOR}. 
In \cite{GGOR}, 
they defined 
the KZ functor $\KZ_n : \CO_{n,r} \ra \He_{n,r} \cmod$. 
Then $\CO_{n,r}$ is the highest weight cover of $\He_{n,r} \cmod$ in the sense of \cite{R} 
through the KZ functor. 
In \cite{R}, 
Rouquier proved that 
$\CO_{n,r}$ is equivalent to 
$\Sc_{n,r} \cmod$ 
as highest weight covers of $\He_{n,r} \cmod$ 
under some conditions for parameters. 

\para 
Let $\OInd^{n+1}_n$ (resp. $\ORes^{n+1}_n$) 
be the parabolic induction (resp. restriction) functors 
between $\CO_{n,r}$ and $\CO_{n+1,r}$ 
defined in \cite{BE}. 
Then, we have the following theorem. 

\begin{thm}
\label{Thm iso Res Ind SCA RCA} 
Assume that 
$\CO_{n,r}$ (resp. $\CO_{n+1,r}$) 
is equivalent to 
$\Sc_{n,r} \cmod$ (resp. $\Sc_{n+1,r} \cmod$) 
as highest weight covers of $\He_{n,r} \cmod$ 
(resp. $\He_{n+1,r} \cmod$). 
Then, under these equivalences, 
we have the following isomorphisms of functors: 
\[ \ORes^{n+1}_n \cong \Res^{n+1}_n, \quad \OInd^{n+1}_n \cong \Ind^{n+1}_n .\]
\end{thm}

\begin{proof}
Let 
$\Theta_n : \CO_{n,r} \ra \Sc_{n,r} \cmod$ 
be the functor giving the equivalence as highest weight covers of $\He_{n,r} \cmod$. 
Then, we have that 
$\KZ_n \cong \Om_n \circ \Theta_n$. 
Let $\Psi_n : \He_{n,r} \cmod \ra \CO_{n,r}$ 
be the right adjoint functor of $\KZ_n$. 
Then, 
we have that 
$\Phi_n \cong \Theta_n \circ \Psi_n$ 
by the uniqueness of the adjoint functors. 
It is similar for the equivalence $\Theta_{n+1} : \CO_{n+1,r} \ra \Sc_{n+1,r} \cmod$.   

By Corollary \ref{Cor properties Res Ind} (\roiv) and \cite[Theorem 2.1]{S}, 
we have that 
\begin{align}
\label{Om Res KZ Res}
\begin{split}
\Om_n \circ \Res^{n+1}_n 
& \cong 
\HRes^{n+1}_n \circ \Om_{n+1} 
\\
& \cong 
\HRes^{n+1}_n \circ \KZ_{n+1} \circ \Theta^{-1}_{n+1} 
\\
& \cong 
\KZ_n \circ \ORes^{n+1}_n \circ \Theta^{-1}_{n+1}.
\end{split} 
\end{align}
Recall that 
$I_{n+1} : \Sc_{n+1,r}\cproj \ra \Sc_{n+1,r} \cmod$ 
is the canonical embedding functor. 
Then, 
\eqref{Om Res KZ Res} together with the isomorphism 
$\Phi_n \cong \Theta_n \circ \Psi_n$ 
implies that 
\begin{align}
\label{Phi Om Res I Psi KZ Res I}
\Phi_n \circ \Om_n \circ \Res^{n+1}_n \circ I_{n+1} 
\cong 
\Theta_n \circ \Psi_n \circ \KZ_n  \circ \ORes^{n+1}_n \circ \Theta^{-1}_{n+1}  \circ I_{n+1}. 
\end{align}
Since $\Res^{n+1}_n$ (resp. $\ORes^{n+1}_n$) 
has the left and right adjoint functor $\Ind^{n+1}_n$ (resp. $\OInd^{n+1}_n$) 
by Corollary \ref{Cor properties Res Ind} (\roii) (resp. \cite[Proposition 2.9]{S}), 
$\Res^{n+1}_n$ (resp. $\ORes^{n+1}_n$) carries projectives to projectives. 
Moreover, 
$\Phi_n \circ \Om_n \cong \Id_n$ (resp. $ \Psi_n \circ \KZ_n \cong \Id_n$) 
on projective objects by properties of highest weight covers (see \cite[Proposition 4.33]{R}). 
Thus, \eqref{Phi Om Res I Psi KZ Res I} implies that 
\begin{align}
\label{Res I Res I}  
\Res^{n+1}_n \circ I_{n+1} 
\cong 
\Theta_n \circ \ORes^{n+1}_n \circ \Theta^{-1}_{n+1}  \circ I_{n+1}
\end{align}
Since both $\Res^{n+1}_n$ and $\Theta_n \circ \ORes^{n+1}_n \circ \Theta^{-1}_{n+1}$ 
are exact, 
\eqref{Res I Res I} together with 
Lemma \ref{Lemma lift morphism of functors} (\roii) 
implies that 
\[ \Res^{n+1}_n \cong \Theta_n \circ \ORes^{n+1}_n \circ \Theta^{-1}_{n+1}. \] 
By the uniqueness of adjoint functors (or by a similar arguments),  
we also have 
\[ \Ind^{n+1}_n \cong \Theta_{n+1} \circ \OInd^{n+1}_n \circ \Theta_{n}^-.\]
\end{proof}


\begin{thebibliography}{DJM10}


\bibitem[AM]{AM} 
S.~Ariki and A.~Mathas, 
\newblock The number of simple modules of the Hecke algebras of type $G(r,1,n)$, 
\newblock {\em Math. Z.} {\bf 233} (2000), 601-623. 


\bibitem[ASS]{ASS} 
I. Assem, D. Simson, A. Skowronski, 
\newblock Elements of the Representation Theory of Associative Algebras, 
London Math. Soc., Student Texts, {\bf 65}, 
Cambridge Univ. Press. 


\bibitem[BE]{BE} 
R.~ Bezrukavnikov and P.~Etingof, 
\newblock Parabolic induction and restriction functors for rational Cherednik algebras,  
\newblock {\em Sel. Math.} {\bf 14} (2009), 397-425. 




\bibitem[DJM]{DJM98}
R.~Dipper, G.~James, and A.~Mathas, 
\newblock Cyclotomic {$q$}-{S}chur algebras, 
\newblock {\em Math. Z.} {\bf 229} (1998), 385-416.

\bibitem[DM]{DM}
R.~Dipper and A.~Mathas,  
\newblock Morita equivalences of Ariki-Koike algebras, 
\newblock {\em Math. Z.}  {\bf 240} (2002), 579-610. 

\bibitem[D]{D-book}
S.~Donkin, 
\newblock The $q$-Schur algebra,  
\newblock London Math. Soc., Lecture Note series {\bf 253}, 
Cambridge Univ. Press. 


\bibitem[DR]{DR}
J.~Du and H.~Rui,  
\newblock  Borel type subalgebras of the $q$-Schur$^m$ algebra,  
\newblock {\em J. Algebra} {\bf 213} (1999), 567-595. 

\bibitem[GGOR]{GGOR}
V.~ Ginzburg, N.~Guay, E.~Opdam and R.~Rouquier, 
\newblock On the category $\CO$ for rational Cherednik algebras, 
\newblock {\em Invent. Math.} {\bf 154} (2003), 617-651. 

\bibitem[JMMO]{JMMO}
M.~Jimbo, K.~Misra, T.~Miwa and M.~Okado, 
\newblock Combinatorics of representations of $U_q (\wh{\Fsl}_n)$ at $q=0$, 
\newblock {\em Comm. Math. Phys.} {\bf 136} (1991), 543-566.

\bibitem[LM]{LM}
S.~Lyle and A.~Mathas, 
\newblock Blocks of cyclotomic Hecke algebras, 
\newblock {\em Adv. Math. } {\bf 216} (2007), 854-878.  

\bibitem[MM]{MM} 
G.~Malle and A.~Mathas, 
\newblock Symmetric cyclotomic Hecke algebras, 
\newblock{\em J. Algebra} {\bf 205} (1998), 275-293. 


\bibitem[M1]{Mat04}
A.~Matahs, 
\newblock The representation theory of the {A}riki-{K}oike and cyclotomic
  {$q$}-{S}chur algebras.
\newblock In \lq\lq {\em Representation theory of algebraic groups and quantum
  groups}", {\em Adv. Stud. Pure Math.} Vol. {\bf 40}, Math. Soc.
  Japan, Tokyo 2004, pp. 261-320.
  
\bibitem[M2]{Mat08}
A.~Mathas, 
\textit{Seminormal forms and {G}ram determinants for cellular algebras}, 
J. Reine Angew. Math. 
{\bf 619} 
(2008), 141--173.


\bibitem[R1]{R}
R.~Rouquier, 
\newblock{ $q$-Schur algebras and complex reflection groups}, 
{\em Moscow Math. J.} {\bf 8}, 119-158.

\bibitem[R2]{R2}
R.~Rouquier, 
\newblock{2-Kac-Moody algebras}, preprint, arXiv: 0812.5023.  

\bibitem[S]{S}
P.~Shan, 
\newblock 
Crystals of Fock spaces and cyclotomic rational double affine Hecke algebras, 
{\em Ann. Sci. Ec. Norm. Super.} {\bf 44} (2011), 147-182. 

\bibitem[SW]{SW} 
C.~Stroppel and B.~Webster, 
\newblock{Quiver Schur algebras and $q$-Fock space,} 
preprint, arXiv:1110.1115.


\bibitem[W]{W}
K.~Wada, 
\newblock{Presenting cyclotomic $q$-Schur algebras}, 
\newblock {\em Nagoya Math. J.} {\bf 201} (2011), 45-116.

\bibitem[W2]{W2}
K.~Wada, 
\newblock{On Weyl modules of cyclotomic $q$-Schur algebras}, 
to appear in Contemp. Math.
\end{thebibliography}
\end{document}